\documentclass[11pt,a4paper,leqno]{amsart}

\title[Semigroups for quadratic evolution equations]
{Semigroups for quadratic evolution equations acting on Shubin-Sobolev and Gelfand-Shilov spaces}

\author[P. Wahlberg]{Patrik Wahlberg}

\address{Dipartimento di Scienze Matematiche, Politecnico di Torino, Corso Duca degli Abruzzi 24,
10129 Torino, Italy}

\email{patrik.wahlberg[AT]polito.it}

\usepackage{latexsym}
\usepackage[a4paper]{geometry}
\usepackage{amsmath}
\usepackage{amssymb}
\usepackage{amsthm}
\usepackage{bbm}
\usepackage{amsfonts}
\usepackage{mathrsfs}
\usepackage{calc}
\usepackage{cite}
\usepackage{color}
\usepackage{graphicx}
\usepackage{enumerate}

\numberwithin{equation}{section}          %Detta gr att man f�r
\newtheorem{thm}{Theorem}
\numberwithin{thm}{section}

\newcommand{\rubrik}{}
\newtheorem{prop}[thm]{Proposition}
\newtheorem{cor}[thm]{Corollary}
\newtheorem{lem}[thm]{Lemma}

\theoremstyle{definition}

\theoremstyle{remark}

\newtheorem{rem}[thm]{Remark}              %T o m hit r bara allmn
\newcommand{\scal}[2]{\langle #1,#2\rangle}
\newcommand{\pd}[1] {\partial ^#1}
\newcommand{\pdd}[2] {\partial_{#1} ^{#2}}

\newcommand{\ro}{\mathbf R}
\newcommand{\no}{\mathbf N}
\newcommand{\rr}[1]{\mathbf R^{#1}}
\newcommand{\nn}[1]{\mathbf N^{#1}}

\newcommand{\zo}{\mathbf Z}

\newcommand{\co}{\mathbf C}
\newcommand{\cc}[1]{\mathbf C^{#1}}

\newcommand{\dd}{\mathrm {d}}

\newcommand{\nm}[2]{\Vert #1\Vert _{#2}}

\newcommand{\ep}{\varepsilon}
\newcommand{\fy}{\varphi}

\newcommand{\cdo}{\, \cdot \, }

\newcommand{\eabs}[1]{\langle #1\rangle}

\newcommand{\Sp}{\operatorname{Sp}}
\newcommand{\ssp}{\operatorname{sp}}
\newcommand{\GL}{\operatorname{GL}}

\newcommand{\On}{\operatorname{O}}

\newcommand{\cS}{\mathscr{S}}

\newcommand{\cF}{\mathscr{F}}

\newcommand{\cK}{\mathscr{K}}

\newcommand{\J}{\mathcal{J}}

\newcommand{\re}{{\rm Re} \,}
\newcommand{\im}{{\rm Im} \,}
\def\la{\langle}
\def\ra{\rangle}
\newcommand{\leqs}{\leqslant}
\newcommand{\geqs}{\geqslant}

\begin{document}

\begin{abstract}
We consider the initial value Cauchy problem for a class of evolution equations whose Hamiltonian is the 
Weyl quantization of a homogeneous quadratic form with non-negative definite real part. 
The solution semigroup is shown to be strongly continuous on several spaces: 
the Shubin--Sobolev spaces, the Schwartz space, the tempered distributions, 
the equal index Beurling type Gelfand--Shilov spaces and their dual ultradistribution spaces. 
\end{abstract}


\maketitle

%%%%%%%%%%%%%%%%%%%%%%%
\section{Introduction}
%%%%%%%%%%%%%%%%%%%%%%%

Consider the Cauchy problem for the evolution equation
\begin{equation*}
	\left\{
	\begin{array}{rl}
	\partial_t u(t,x) + q^w(x,D) u (t,x) & = 0, \qquad t > 0, \quad x \in \rr d, \\
	u(0,\cdot) & = u_0\in L^2(\rr d), 
	\end{array}
	\right.
\end{equation*}
where $q^w(x,D)$ is the Weyl quantization of a symbol $q$ which is a homogeneous quadratic form on the phase space $T^* \rr d$, 
defined by a symmetric matrix $Q \in \cc {2d \times 2d}$ such that ${\re} Q \geqslant 0$. 
Particular cases include the heat equation, the free Schr\"odinger equation and the harmonic oscillator Schr\"odinger equation. 

H\"ormander \cite{Hormander2} showed that the solution operator $e^{- t q^w(x,D)}$ is a strongly continuous contraction semigroup on $L^2 (\rr d)$ with respect to the parameter $t \geqs 0$. 
Semigroup theory then guarantees that $u(t,x) = e^{- t q^w(x,D)} u_0$ is the unique solution to the Cauchy problem
when $u_0 \in D(q^w(x,D)) \subseteq L^2(\rr d)$ where $D(q^w(x,D))$ denotes the domain of the closure of $q^w(x,D)$ considered as an unbounded operator on $L^2$. 
In this paper we show that the semigroup $e^{- t q^w(x,D)}$ is strongly continuous in several other functional frameworks. 

First we show strong continuity on the Shubin--Sobolev spaces, or Hilbert modulation spaces $M_s^2(\rr d)$, with polynomial weights indexed by $s \in \ro$. 
Since the $M_s^2(\rr d)$ norms for $s \geqs 0$ is a system of seminorms for the Schwartz space $\cS(\rr d)$ we obtain as byproduct the following results. 
The propagator $e^{- t q^w(x,D)}$ is a locally equicontinuous strongly continuous semigroup on $\cS(\rr d)$. 
By duality it is also 
strongly continuous on the tempered distributions $\cS' (\rr d)$, equipped with either the weak$^*$ or the strong topology. 
In the latter case the semigroup is moreover locally equicontinuous. 

Then we consider the equal index Beurling type Gelfand--Shilov spaces $\Sigma_s(\rr d)$ for $s > \frac{1}{2}$. 
Again we prove that the propagator is a locally equicontinuous strongly continuous semigroup on $\Sigma_s (\rr d)$, 
that extends by duality to a strongly continuous semigroup on the Gelfand--Shilov ultradistribution space $\Sigma_s' (\rr d)$, equipped with either the weak$^*$ or the strong topology.
In the latter case we show again local equicontinuity. 
In the process we show that the Gelfand--Shilov space $\Sigma_s(\rr d)$ is reflexive, which apparently has not been stated in the literature. 

The proofs rely heavily on H\"ormander's results \cite{Hormander2}. 
We use both his formula for the Weyl symbol of the propagator $e^{- t q^w(x,D)}$, and his expression of the propagator as a Fourier integral operator with respect to a quadratic phase function.  
The latter is a particular case of an extension of the metaplectic group, called the metaplectic semigroup in \cite{Hormander2}, indexed by the semigroup of complex symplectic matrices that are positive in a certain sense. 

The results presented here provide a link that is missing in our papers \cite{Carypis1,PRW1,Wahlberg1}. 
In fact a discussion on the action of the solution semigroup on tempered distributions and on Gelfand--Shilov ultradistributions is lacking in them. 

The class of evolution equations under study in this paper is currently an active field of research
\cite{Hitrik1,Ottobre1,PRW1}. 
In particular it has been studied with respect to Gelfand--Shilov smoothing effects
\cite{Hitrik2,Hitrik3}, where it turns out that the singular space \cite{Hitrik1} plays a crucial role. 
The singular space is a linear subspace of the phase space $T^* \rr d$ determined by the quadratic form $q$. 

The paper is organized as follows. 
Section \ref{sec:prel} treats the functional analytical background concerning the spaces of functions and (ultra-)distributions we study. 
In Section \ref{sec:evolution} we specify the investigated class of evolution equations, and we give a brief overview of 
H\"ormander's results \cite{Hormander2} on the propagator acting on $L^2$ expressed with Fourier integral operators.  
In Section \ref{sec:propagator} we prepare for the main results in Sections \ref{sec:strongcontmodulation} and \ref{sec:GelfandShilov}, 
in particular by using results from \cite{Hormander2} to study the action of differential and monomial multiplication operators to the left of the propagator. 
Section \ref{sec:strongcontmodulation} treats strong continuity of the semigroup on Shubin--Sobolev spaces and its consequences, 
and finally Section \ref{sec:GelfandShilov} concerns strong continuity on Gelfand--Shilov spaces and their duals. 

%%%%%%%%%%%%%%%%%%%%%%%
\section{Preliminaries}\label{sec:prel}
%%%%%%%%%%%%%%%%%%%%%%%

An open ball in a Banach space $X$ with center $x_0 \in X$ and radius $r > 0$ is denoted $B_r (x_0) = \{ x \in X: \ \| x - x_0 \| < r\}$, 
and $B_r = B_r (0)$. 
We use $\eabs{x} = (1+|x|^2)^{\frac{1}{2}}$ for $x \in \rr d$,
and the partial derivative $D_j = - i \partial_j$, $1 \leqs j \leqs d$, acting on functions and distributions on $\rr d$, 
with extension to multi-indices. 
The standard basis vector in $\rr d$ with index $1 \leqs j \leqs d$ is denoted $e_j \in \rr d$. 
The transpose of a matrix $A \in \cc {d \times d}$ is denoted $A^T$. 
The real (complex) quadratic matrices of dimension $d$ is $\rr {d \times d}$ ($\cc {d \times d}$), the group of invertible real (complex) matrices is denoted $\GL(d,\ro) \subseteq \rr {d \times d}$ ($\GL(d,\co) \subseteq \cc {d \times d}$), and the subgroup of real orthogonal matrices is denoted $\On(d) \subseteq \GL(d,\ro)$.  

We write $f (x) \lesssim g (x)$ provided there exists $C>0$ such that $f (x) \leqs C \, g(x)$ for all $x$ in the domain of $f$ and of $g$. 
The symbol $f (x) \asymp g (x)$ means that $f (x) \lesssim g (x)$ and $g (x) \lesssim f (x)$. 
The normalization of the Fourier transform is
\begin{equation*}
 \cF f (\xi )= \widehat f(\xi ) = (2\pi )^{-\frac d2} \int _{\rr
{d}} f(x)e^{-i\scal  x\xi }\, \dd x, \qquad \xi \in \rr d, 
\end{equation*}
for $f\in \cS(\rr d)$ (the Schwartz space), where $\scal \cdo \cdo$ denotes the scalar product on $\rr d$. 
The conjugate linear action of a (ultra-)distribution $u$ on a test function $\phi$ is written $(u,\phi)$, consistent with the $L^2$ inner product $(\cdo ,\cdo ) = (\cdo ,\cdo )_{L^2}$ which is conjugate linear in the second argument. 

Denote translation by $T_x f(y) = f( y-x )$ and modulation by $M_\xi f(y) = e^{i \scal y \xi} f(y)$ 
for $x,y,\xi \in \rr d$ where $f$ is a function or distribution defined on $\rr d$. 
The composition is denoted $\Pi(x,\xi) = M_\xi T_x$. 
Let $\fy \in \cS(\rr d) \setminus \{0\}$. 
The short-time Fourier transform of a tempered distribution $u \in \cS'(\rr d)$ is defined by 
\begin{equation*}
V_\fy u (x,\xi) = (2\pi )^{-\frac d2} (u, M_\xi T_x \fy), \quad x,\xi \in \rr d. 
\end{equation*}
Then $V_\fy u$ is smooth and polynomially bounded \cite[Theorem~11.2.3]{Grochenig1}, and we have 
\begin{equation}\label{eq:STFTinverse}
(u, f) = (V_\fy u, V_\fy f)_{L^2(\rr {2d})}
\end{equation}
for $u \in \cS'(\rr d)$ and $f \in \cS(\rr d)$, provided $\| \fy \|_{L^2} = 1$, cf. \cite[Theorem~11.2.5]{Grochenig1}. 

The Hilbert modulation space, also known as the Shubin--Sobolev space, $M_s^2(\rr d) \subseteq \cS'(\rr d)$ of order $s \in \ro$\cite{Feichtinger1,Grochenig1,Nicola1,Shubin1} has norm
\begin{equation}\label{eq:Qsnorm}
\| u \|_{M_s^2} := \| \eabs{\cdot}^s V_\varphi u \|_{L^2(\rr {2d})} = \left( \iint_{\rr {2d}} \eabs{(x,\xi)}^{2s} |V_\varphi u(x,\xi)|^2 \, \dd x \, \dd \xi \right)^{1/2}. 
\end{equation}
Different functions $\fy \in \cS (\rr d) \setminus \{ 0 \}$ give equivalent norms. 
We have $M_0^2(\rr d) = L^2(\rr d)$, and for any $s,t \in \ro$ with $t \leqs s$ the embeddings
\begin{equation}\label{eq:shubinsobolevnested}
\cS(\rr d) \subseteq M_{s}^2 (\rr d) \subseteq M_t^2(\rr d) \subseteq \cS'(\rr d) 
\end{equation}
where $\cS'$ is equipped with its weak$^*$ topology, 
and 
\begin{equation}\label{eq:shubinsobolevspaces}
\cS(\rr d) = \bigcap_{s \in \ro} M_s^2 (\rr d), \qquad \cS'(\rr d) = \bigcup_{s \in \ro} M_s^2(\rr d). 
\end{equation}
(Inclusions of function and distribution spaces understand embeddings.)

We need some elements from the calculus of pseudodifferential operators \cite{Folland1,Hormander0,Nicola1,Shubin1}. 
Let $a \in C^\infty (\rr {2d})$ and $m \in \ro$. Then $a$ is a \emph{Shubin symbol} of order $m$, denoted $a\in \Gamma^m$, if for all $\alpha,\beta \in \nn d$ there exists a constant $C_{\alpha,\beta}>0$ such that
\begin{equation}\label{eq:shubinineq}
|\partial_x^\alpha \partial_\xi^\beta a(x,\xi)| \leqs C_{\alpha,\beta} \langle (x,\xi)\rangle^{m-|\alpha + \beta|}, \quad x,\xi \in \rr d.
\end{equation}
The Shubin symbols $\Gamma^m$ form a Fr\'echet space where the seminorms are given by the smallest possible constants in \eqref{eq:shubinineq}.

For $a \in \Gamma^m$ a pseudodifferential operator in the Weyl quantization is defined by
\begin{equation}\label{eq:weylquantization}
a^w(x,D) f(x)
= (2\pi)^{-d}  \int_{\rr {2d}} e^{i \langle x-y, \xi \rangle} a \left(\frac{x+y}{2},\xi \right) \, f(y) \, \dd y \, \dd \xi, \quad f \in \cS(\rr d),
\end{equation}
when $m<-d$. The definition extends to general $m \in \ro$ if the integral is viewed as an oscillatory integral.
The operator $a^w(x,D)$ then acts continuously on $\cS(\rr d)$ and extends uniquely by duality to a continuous operator on $\cS'(\rr d)$.
By Schwartz's kernel theorem the Weyl quantization procedure may be extended to a weak formulation which yields operators $a^w(x,D):\cS(\rr{d}) \to \cS'(\rr{d})$, even if $a$ is only an element of $\cS'(\rr{2d})$.

For $a \in \cS'(\rr {2d})$ and $f,g \in \cS(\rr d)$ we have
\begin{equation}\label{eq:wignerweyl}
(a^w(x,D) f,g) = (2 \pi)^{-\frac{d}{2}} (a, W(g,f) ) 
\end{equation}
where 
\begin{equation}\label{eq:wignerdistribution}
W(g,f) (x,\xi) = (2 \pi)^{-\frac{d}{2}} \int_{\rr d} g(x+y/2) \overline{f(x-y/2)} \, e^{- i \la y, \xi \ra} \, \dd y \in \cS(\rr {2d})
\end{equation}
is the Wigner distribution \cite{Folland1,Grochenig1}.

According to \cite[Theorem~1.7.16 and Corollary 1.7.17]{Nicola1}, \cite[Theorem~25.2]{Shubin1}, the Weyl operators with $a\in \Gamma^m$ act continuously on the Hilbert modulation spaces as
\begin{equation}\label{eq:shubincontMs}
a^w(x,D): M_{s}^2(\rr d) \to M_{s-m}^2(\rr d), \quad s \in \ro. 
\end{equation}

The real phase space $T^* \rr d \simeq \rr d \oplus \rr d$ is a real symplectic vector space equipped with the 
canonical symplectic form
\begin{equation*}
\sigma((x,\xi), (x',\xi')) = \langle x' , \xi \rangle - \langle x, \xi' \rangle, \quad (x,\xi), (x',\xi') \in T^* \rr d. 
\end{equation*}
This form can be expressed with the inner product as $\sigma(X,Y) = \la \J X, Y \ra$ for $X,Y \in T^* \rr d$
where 
\begin{equation}\label{eq:Jdef}
\J =
\left(
\begin{array}{cc}
0 & I_d \\
-I_d & 0
\end{array}
\right) \in \rr {2d \times 2d}. 
\end{equation}
The complex phase space $T^* \cc d \simeq \cc d \oplus \cc d$ is likewise a complex symplectic vector space with respect to the same symplectic form. 
(Note that $\scal \cdo \cdo$ is not conjugate linear in one argument, but bilinear for arguments in $\cc d \times \cc d$.)
The real (complex) symplectic group $\Sp(d,\ro)$ ($\Sp(d,\co)$) is the set of matrices in $\GL(2d,\ro)$ ($\GL(2d,\co)$) that leaves $\sigma$ invariant. 
Hence $\J \in \Sp(d,\ro)$. 
A Lagrangian subspace $\lambda \subseteq T^* \rr d$ ($\lambda \subseteq T^* \cc d$) is a real (complex) linear space of dimension $d$ such that $\sigma |_{\lambda \times \lambda} = 0$. 
A Lagrangian $\lambda \subseteq T^* \cc d$ is called positive \cite{Hormander1,Hormander2} if 
\begin{equation*}
i \sigma( \overline X, X) \geqs 0, \quad X \in \lambda. 
\end{equation*}

To each symplectic matrix $\chi \in \Sp(d,\ro)$ is associated an operator $\mu(\chi)$ that is unitary on $L^2(\rr d)$, and determined up to a complex factor of modulus one, such that
\begin{equation}\label{symplecticoperator}
\mu(\chi)^{-1} a^w(x,D) \, \mu(\chi) = (a \circ \chi)^w(x,D), \quad a \in \cS'(\rr {2d})
\end{equation}
(cf. \cite{Folland1,Hormander0}).
The operator $\mu(\chi)$ is a homeomorphism on $\mathscr S$ and on $\mathscr S'$.

The mapping $\Sp(d,\ro) \ni \chi \rightarrow \mu(\chi)$ is called the metaplectic representation \cite{Folland1}.
It is in fact a representation of the so called $2$-fold covering group of $\Sp(d,\ro)$, which is called the metaplectic group.  
The metaplectic representation satisfies the homomorphism relation modulo a change of sign:
\begin{equation*}
\mu( \chi \chi') = \pm \mu(\chi ) \mu(\chi' ), \quad \chi, \chi' \in \Sp(d,\ro).
\end{equation*}

We will use two systems of seminorms on $\cS ( \rr d)$. 
The first is 
\begin{equation}\label{eq:seminormsS1}
\cS \ni \fy \mapsto \| \fy \|_n := \max_{|\alpha+\beta| \leqs n} \sup_{x \in \rr d} \left| x^\alpha D^\beta \fy (x) \right|, \quad n \in \no, 
\end{equation}
and the second is 
\begin{equation}\label{eq:seminormsS2}
\cS \ni \fy \mapsto \| \fy \|_{M_s^2}, \quad s \geqs 0.  
\end{equation}
The fact that the seminorms \eqref{eq:seminormsS2} are equivalent to \eqref{eq:seminormsS1}
follows from \cite[Corollary~11.2.6 and Lemma~11.3.3]{Grochenig1}.

Let $h,s > 0$ be fixed. The space denoted $\mathcal S_{s,h}(\rr d)$
is the set of all $f\in C^\infty (\rr d)$ such that
\begin{equation}\label{eq:seminormSigmas}
\nm f{\mathcal S_{s,h}}\equiv \sup \frac {|x^\alpha D ^\beta
f(x)|}{h^{|\alpha + \beta |}(\alpha !\, \beta !)^s}
\end{equation}
is finite, where the supremum is taken over all $\alpha ,\beta \in
\mathbf N^d$ and $x\in \rr d$.
The function space $\mathcal S_{s,h}$ is a Banach space which increases
with $h$ and $s$, and $\mathcal S_{s,h} \subseteq \cS$.
The topological dual $\mathcal S_{s,h}'(\rr d)$ is
a Banach space and $\cS'(\rr d) \subseteq \mathcal S_{s,h}'(\rr d)$.
If $s>1/2$ then $\mathcal
S_{s,h}$ and $\bigcup _{h>0} \mathcal S _{1/2,h}$ contain all finite linear
combinations of Hermite functions.

The Beurling type \emph{Gelfand--Shilov space}
$\Sigma _s(\rr d)$ is the projective limit 
of $\mathcal S_{s,h}(\rr d)$ with respect to $h$ \cite{Gelfand2}. This means
\begin{equation}\label{GSspacecond1}
\Sigma _{s}(\rr d) = \bigcap _{h>0} \mathcal S_{s,h}(\rr d)
\end{equation}
and the Fr{\'e}chet space topology of $\Sigma _s(\rr d)$ is defined by the seminorms $\nm \cdot{\mathcal S_{s,h}}$ for $h>0$.
Then $\Sigma _s(\rr d)\neq \{ 0\}$ if and only if $s>1/2$ \cite{Petersson1}.
The topological dual of $\Sigma _s(\rr d)$ is the space of (Beurling type) \emph{Gelfand--Shilov ultradistributions} \cite[Section~I.4.3]{Gelfand2}
\begin{equation}\tag*{(\ref{GSspacecond1})$'$}
\Sigma _s'(\rr d) =\bigcup _{h>0} \mathcal S_{s,h}'(\rr d).
\end{equation}

The dual space $\Sigma _s'(\rr d)$ may be equipped with several topologies: the weak$^*$ topology, the strong topology, the Mackey topology, and the topology defined by the union \eqref{GSspacecond1}$'$ as an inductive limit topology \cite{Schaefer1}. The latter topology is the strongest topology such that the inclusion $\mathcal S_{s,h}'(\rr d) \subseteq \Sigma _s'(\rr d)$ is continuous for all $h > 0$.  

As we shall see shortly, the space $\Sigma _{s}(\rr d)$ may be equipped with Hilbert space seminorms, and thus 
it may be considered a countably-Hilbert space \cite{Becnel1}. According to \cite[Theorem~4.16]{Becnel1}
the strong, the Mackey and the inductive limit topologies on $\Sigma _s'(\rr d)$ coincide. 

We will study $\Sigma_s'(\rr d)$ equipped with the weak$^*$ topology, denoted $\Sigma_{s, \rm w}'(\rr d)$, or with the strong topology, denoted $\Sigma_{s,\rm str}'(\rr d)$. 
The latter topology is defined by seminorms 
\begin{equation*}
\Sigma _s'(\rr d) \ni u \mapsto \sup_{\fy \in B} |(u,\fy)|
\end{equation*}
for each subset $B \subseteq \Sigma_s(\rr d)$ which is bounded, that is uniformly bounded with respect to each seminorm. 
Both spaces $\Sigma_{s,\rm w}'(\rr d)$ and $\Sigma_{s,\rm str}'(\rr d)$ are sequentially complete \cite[Theorems~I.5.1 and I.5.6]{Gelfand2}. 
From the latter result we also have: A sequence is convergent in $\Sigma_{s,\rm w}'(\rr d)$ exactly when it converges in the weak$^*$ topology of $\mathcal S_{s,h}'(\rr d)$ for some $h>0$. 

By the proof of Proposition \ref{prop:GSreflexive} (see Section \ref{sec:GelfandShilov}) it will follow that the space $\Sigma_s(\rr d)$ is a \emph{perfect} space in the terminology of \cite{Gelfand2}: It is a space in which any bounded set is relatively compact. 
By \cite[Theorem~I.6.4]{Gelfand2} sequential convergence in $\Sigma_{s,\rm w}'$ and $\Sigma_{s,\rm str}'$ hence coincide.

The Roumieu type Gelfand--Shilov space is the union 
\begin{equation*}
\mathcal S_s(\rr d) = \bigcup _{h>0}\mathcal S_{s,h}(\rr d)
\end{equation*}
equipped with the inductive limit topology \cite{Schaefer1}, that is 
the strongest topology such that each inclusion $\mathcal S_{s,h}(\rr d) \subseteq\mathcal S_s(\rr d)$
is continuous. 
Then $\mathcal S _s(\rr d)\neq \{ 0\}$ if and only if $s \geqs 1/2$. 
The corresponding (Roumieu type) Gelfand--Shilov ultradistribution space is 
\begin{equation*}
\mathcal S_s'(\rr d) = \bigcap _{h>0}\mathcal S_{s,h}'(\rr d). 
\end{equation*}
For every $s > 0$ and $\ep > 0$
\begin{equation*}
\Sigma _s (\rr d)\subseteq \mathcal S_s(\rr d)\subseteq
\Sigma _{s+\ep}(\rr d).
\end{equation*}
We will not use the Roumieu type spaces in this article but mention them as a service to a reader interested in a wider context. 
On a similar note we notice that $(\alpha! \beta!)^s$ in \eqref{eq:seminormSigmas} may be replaced 
by $\alpha!^{s_1} \beta!^{s_2}$ for different parameters $s_1, s_2 > 0$ which leads to a more flexible family of spaces. 
In this paper we restrict to the equal index case.

The Gelfand--Shilov (ultradistribution) spaces enjoy invariance properties, with respect to 
translation, dilation, tensorization, coordinate transformation and (partial) Fourier transformation.
The Fourier transform extends 
uniquely to homeomorphisms on $\mathscr S'(\rr d)$, $\mathcal
S_s'(\rr d)$ and $\Sigma _s'(\rr d)$, and restricts to 
homeomorphisms on $\mathscr S(\rr d)$, $\mathcal S_s(\rr d)$
and $\Sigma _s(\rr d)$, and to a unitary operator on $L^2(\rr d)$.
In particular the Wigner distribution \eqref{eq:wignerdistribution} satisfies $W(g,f) \in \Sigma_s(\rr {2d})$
if $f,g \in \Sigma_s(\rr d)$, and the Weyl quantization formula \eqref{eq:wignerweyl} holds for $a \in \Sigma_s'(\rr {2d})$ and $f,g \in \Sigma_s(\rr d)$. 
Likewise \eqref{eq:STFTinverse} holds when $u \in \Sigma_s'(\rr d)$, $f \in \Sigma_s(\rr d)$, $\fy \in \Sigma_s(\rr d)$ and $\| \fy \|_{L^2} = 1$. 

We will use the Hermite functions 
\begin{equation*}
h_\alpha (x) = \pi ^{-\frac d4}(-1)^{|\alpha |}
(2^{|\alpha |}\alpha !)^{-\frac 12}e^{\frac {|x|^2}2}
\partial ^\alpha e^{-|x|^2}, \quad x \in \rr d, \quad \alpha \in \nn d,  
\end{equation*}
and formal series expansions with respect to Hermite functions: 
\begin{equation*}
f = \sum_{\alpha \in \nn d} c_\alpha h_\alpha
\end{equation*}
where $\{ c_\alpha \}$ is a sequence of complex coefficients
defined by $c_\alpha = c_\alpha(f) = (f,h_\alpha)$. 

Gelfand--Shilov spaces and their ultradistribution
duals, as well as the Schwartz space $\cS$ and the tempered distributions $\cS'$, and $L^2$,
can be identified by means of such series expansions, 
with characterizations in terms of the corresponding sequence
spaces (see \cite{Gramchev1,Gramchev2,Langenbruch1,Reed1}). 
Let
\begin{equation*}
f = \sum_{\alpha \in \nn d} c_\alpha h_\alpha 
\end{equation*}
and 
\begin{equation*}
\phi = \sum_{\alpha \in \nn d} d_\alpha h_\alpha 
\end{equation*}
with sequences $\{ c_\alpha\}$ and $\{d_\alpha \}$ of finite support.
Then the sesquilinear form
\begin{equation}\label{eq:sequilinearform}
( f, \phi ) = \sum_{\alpha \in \nn d} c_\alpha \overline{d_\alpha} 
\end{equation}
agrees with the inner product on $L^2(\rr d)$ since
$\{ h_\alpha \}_{\alpha \in \nn d} \subseteq L^2(\rr d)$ is an orthonormal basis. 

The form \eqref{eq:sequilinearform} extends uniquely to the duality on $\cS'(\rr d) \times \cS (\rr d)$, 
to the duality on $\mathcal S_{s}'(\rr d) \times \mathcal S_{s}(\rr d)$ for $s \geqs 1/2$, 
as well as to the duality on $\Sigma_{s}'(\rr d) \times \Sigma_{s}(\rr d)$ for $s > 1/2$.

To wit Simon \cite[Theorem~V.13]{Reed1} showed
that the family of Hilbert sequence spaces
\begin{equation*}
\ell_{r}^2 = \ell_{r}^2(\nn d)  
= \left\{ \{c_\alpha\}: \| c_\alpha \|_{\ell_r^2} = 
\left( \sum_{\alpha \in \nn d} |c_\alpha|^2 \eabs{\alpha}^{2 r}  \right)^{\frac{1}{2}} < \infty \right\}
\end{equation*}
for $r > 0$ provides a family of seminorms for $\cS$ that is equivalent to \eqref{eq:seminormsS1}, 
via the homeomorphism $\cS \ni f \mapsto \{ (f,h_\alpha)\}_{\alpha \in \nn d}$. 
Thus the Schwartz space $\cS (\rr d)$ is identified topologically
as the projective limit 
\begin{equation}\label{eq:Ssequence}
\cS(\rr d) = \bigcap_{r>0} \left\{ \sum_{\alpha \in \nn d}
c_\alpha h_\alpha : \ \{c_\alpha\} \in \ell_{r}^2 \right\}
\end{equation}
and $\cS'(\rr d)$ is identified \cite[Theorem~V.14]{Reed1} as the union
\begin{equation*}
\cS'(\rr d) = \bigcup_{r>0} \left\{
\sum_{\alpha \in \nn d} c_\alpha h_\alpha : \
{\{c_\alpha\} \in \ell_{-r}^2} \right\}
\end{equation*}
with weak$^*$ convergence of the sum for each element in $\cS'$. 

Likewise Langenbruch \cite[Theorem~3.4]{Langenbruch1} has shown
that the family of Hilbert sequence spaces
\begin{equation*}
\ell_{s,r}^2 = \ell_{s,r}^2(\nn d)  
= \left\{ \{c_\alpha\}: \| c_\alpha \|_{\ell_{s,r}^2} = 
\left( \sum_{\alpha \in \nn d} |c_\alpha|^2 e^{2 r |\alpha|^{\frac{1}{2s}}} \right)^{\frac{1}{2}} < \infty \right\}
\end{equation*}
for $r > 0$ yields a family of seminorms that is equivalent to the family \eqref{eq:seminormSigmas} for all $h > 0$, when $s \geqs \frac1{2}$. 
For $s > 1/2$ this means that the space $\Sigma_{s}(\rr d)$ can be identified topologically
as the projective limit 
\begin{equation}\label{eq:GSsequence}
\Sigma_{s}(\rr d) = \bigcap_{r>0} \left\{ \sum_{\alpha \in \nn d}
c_\alpha h_\alpha : \ \{c_\alpha\} \in \ell_{s,r}^2 \right\} 
\end{equation}
and for $s \geqs 1/2$ the space $\mathcal S_{s}(\rr d)$ can be identified topologically
as the inductive limit 
\begin{equation*}
\mathcal S_{s}(\rr d) = \bigcup_{r>0} \left\{ \sum_{\alpha \in \nn d}
c_\alpha h_\alpha : \ \{c_\alpha\} \in \ell_{s,r}^2 \right\}. 
\end{equation*}

Moreover \cite[Corollary~3.5]{Langenbruch1} shows, in particular, that $\Sigma_{s}'(\rr d)$ may be identified as the union 
\begin{equation*}
\Sigma_{s}'(\rr d) = \bigcup_{r>0} \left\{
\sum_{\alpha \in \nn d} c_\alpha h_\alpha : \
{\{c_\alpha\} \in \ell_{s,-r}^2} \right\}, 
\end{equation*}
and $\mathcal S_{s}'(\rr d)$ may be identified as the intersection 
\begin{equation*}
\mathcal S_{s}'(\rr d) = \bigcap_{r>0} \left\{
\sum_{\alpha \in \nn d} c_\alpha h_\alpha : \
{\{c_\alpha\} \in \ell_{{s,-r}}^2} \right\}, 
\end{equation*}
in both cases with weak$^*$ convergence of the sum for each ultradistribution. 

Working with Gelfand--Shilov spaces we will occasionally need the inequality (cf. \cite{Carypis1})
\begin{equation*}
|x+y|^{1/s} \leqs 2 ( |x|^{1/s} + |y|^{1/s}), \quad x,y \in \rr d, 
\end{equation*}
which holds when $s \geqs \frac{1}{2}$ and which implies 
\begin{equation}\label{eq:exppeetre1}
\begin{aligned}
e^{A |x+y|^{1/s} } & \leqslant e^{2A |x|^{1/s}} e^{2A |y|^{1/s}}, \quad A >0, \quad x,y \in \rr d, \\
e^{- 2 A |x+y|^{1/s} } & \leqslant e^{- A |x|^{1/s}} e^{ 2 A |y|^{1/s}}, \quad A >0, \quad x,y \in \rr d. 
\end{aligned}
\end{equation}

Finally we state the basic definitions of a one-parameter semigroup of operators. 
Often semigroups of operators are considered on a Banach space \cite{Engel1,Pazy1} but we need also the case of a locally convex space \cite{Komura1,Yosida1}. 
Thus let $X$ be a locally convex topological vector space, and let $\{ T_t, \ t \geqs 0 \}$ be a one-parameter family of continuous linear operators on $X$. 
The family $\{ T_t, \ t \geqs 0 \}$ is called a strongly continuous semigroup provided
\begin{equation*}
T_0 = I, \quad
T_t T_s  = T_{t+s}, \quad t,s \geqs 0, \quad {\rm and} \quad 
\lim_{t \to 0^+} T_t x = x \quad \forall x \in X.
\end{equation*}
The infinitesimal generator $A$ of the semigroup $T_t$ is the linear, in general unbounded, operator
\begin{equation*}
A x = \lim_{t \to 0^+} t^{-1} (T_t - I ) x
\end{equation*}
equipped with the domain $D(A) \subseteq X$ of all $x \in X$ such that the right-hand side limit is well defined in $X$. 

A \emph{locally equicontinuous} strongly continuous semigroup \cite{Komura1} is a strongly continuous semigroup $\{ T_t \}_{t \geqs 0}$ on $X$ such that for all $t_0>0$ and each seminorm $p$ on $X$ there exists a seminorm $q$ on $X$ such that 
\begin{equation*}
p (T_t x) \leqs q(x), \quad x \in X, \quad 0 \leqs t \leqs t_0. 
\end{equation*}
%

%%%%%%%%%%%%%%%%%%%%%%%%%%%%%%%%%%
\section{A class of evolution equations and the propagator on $L^2$}
\label{sec:evolution}
%%%%%%%%%%%%%%%%%%%%%%%%%%%%%%%%%%

Let $q$ be a homogeneous quadratic form on $T^* \rr d$, that is 
\begin{equation}\label{eq:quadraticform}
q(x,\xi) = \la (x,\xi), Q(x, \xi) \ra, \quad (x,\xi) \in T^* \rr d, 
\end{equation}
where $Q \in \cc {2d \times 2d}$ is symmetric, and suppose its real part is non-negative definite, denoted ${\re} Q \geqslant 0$. 
We study the initial value Cauchy problem for the following class of evolution equations. 
\begin{equation}
\label{eq:CP}
\tag{CP}
	\left\{
	\begin{array}{rl}
	\partial_t u(t,x) + q^w(x,D) u (t,x) & = 0, \qquad t > 0, \quad x \in \rr d, \\
	u(0,\cdot) & = u_0\in L^2(\rr d). 
	\end{array}
	\right.
\end{equation}
Here $q^w(x,D)$ acts on functions of the variable $x \in \rr d$. 
The \emph{Hamilton map} $F$ corresponding to $q$ is 
\begin{equation*}
F = \J Q \in \cc {2d \times 2d}
\end{equation*}
with $\J \in \Sp(d,\ro)$ defined by \eqref{eq:Jdef}. 
This framework of evolution equations has been studied in many papers, e.g. \cite{Hitrik1,Hormander2,Ottobre1}. 

The symbol $q$ is a Shubin symbol of order two, $q \in \Gamma^2$, 
which implies that $q^w(x,D): M_{s+2}^2(\rr d) \to M_s^2(\rr d)$ is continuous for all $s \in \ro$ by \eqref{eq:shubincontMs}. 
There is a loss of regularity of order two. 

The operator $q^w(x,D)$ can be considered as an unbounded operator in $L^2(\rr d)$. 
In \cite[pp.~425--26]{Hormander2} it is shown that its maximal realization equals its closure as an operator initially defined on $\cS$, 
and the closure of $-q^w(x,D)$ generates a strongly continuous contraction semigroup on $L^2$ for $t \geqs 0$ denoted by $e^{- t q^w(x,D)}$. 
The contraction property means that the $L^2$ operator norm satisfies $\| e^{- t q^w(x,D)} \| \leqs 1$ for all $t \geqs 0$. 

By semigroup theory (see e.g. \cite[Theorem~I.2.4]{Pazy1} and \cite[pp.~483--84]{Kato1}) the unique solution in the space $C^1([0,\infty),L^2)$ to \eqref{eq:CP}
is 
\begin{equation*}
u(x,t) = e^{- t q^w(x,D)} u_0
\end{equation*}
where $u_0 \in D(q^w(x,D)) \subseteq L^2(\rr d)$ which denotes the domain of the closure of $q^w(x,D)$. 
The notation $C^1([0,\infty),L^2)$ understands that the derivative is right continuous at $t=0$. 

In the particular case when ${\re} Q=0$ the propagator is given by means of the metaplectic representation. 
In fact, 
then $e^{-t q^w(x,D)}$ is a group of unitary operators on $L^2(\rr d)$, and we have by \cite[Theorem 4.45]{Folland1}
\begin{equation*}
e^{-t q^w(x,D)} = \mu(e^{-2 i t F}), \quad t \in \ro. 
\end{equation*}
In this case $F$ is purely imaginary and $i F \in \ssp(d,\ro)$, the real symplectic Lie algebra, which implies that $e^{-2 i t F} \in \Sp(d,\ro)$ for any $t \in \ro$ \cite{Folland1}.

In the general case ${\re} Q \geqslant 0$, 
H\"ormander \cite{Hormander2} has shown that the propagator $e^{- t q^w(x,D)}$ can be identified as a time-indexed family of Fourier integral operators, 
described briefly as follows. 
According to \cite[Theorem 5.12]{Hormander2} the Schwartz kernel of the propagator $e^{-t q^w(x,D)}$ for $t \geqs 0$ is an oscillatory integral defined by a quadratic phase function.   
More precisely we have
\begin{equation*}
e^{-t q^w(x,D)} = \cK_{e^{-2 i t F}},
\end{equation*}
where $\cK_{e^{-2 i t F}}: \cS(\rr d) \to \cS'(\rr d) $ is the linear continuous operator with kernel 
\begin{equation}\label{schwartzkernel1}
K_{e^{-2 i t F}} (x,y) 
= (2 \pi)^{-(d+N)/2} \sqrt{\det \left( 
\begin{array}{ll}
p_{\theta \theta}''/i & p_{\theta y}'' \\
p_{x \theta}'' & i p_{x y}'' 
\end{array}
\right) } \int_{\rr N} e^{i p(x,y,\theta)} \dd \theta \in \cS'(\rr {2d}),  
\end{equation}
where the quadratic form $p$ is specified below. 

By \cite[Proposition~5.8]{Hormander2} $\cK_{e^{-2 i t F}}$ is in fact continuous on $\cS (\rr d)$.
The kernel $K_{e^{-2 i t F}}$ is indexed by the matrix $e^{-2 i t F} \in \cc {2d \times 2d}$. 
By \cite[Lemma~5.2]{PRW1} the matrix $e^{-2 i t F}$ belongs to $\Sp(d,\co)$, and its graph
\begin{equation}\label{graph1}
\lambda' : = \mathcal G(e^{- 2 i t F})  = \{(e^{- 2 i t F} X, X): X \in T^* \cc d\} \subseteq T^* \cc d \times T^* \cc d,
\end{equation}
is a positive Lagrangian with respect to the symplectic form $\sigma_1$ defined by \cite[Eq.~(5.1)]{PRW1}. 
As explained after \cite[Lemma~5.1]{PRW1} the Lagrangian $\lambda'$ can be twisted as in \cite[Eq.~(5.2)]{PRW1} to give a positive Lagrangian $\lambda \subseteq T^* \cc {2d}$. 
According to \cite[Theorem 5.12 and p.~444]{Hormander2} the oscillatory integral \eqref{schwartzkernel1} is associated with the positive Lagrangian $\lambda$. 

By \cite[Proposition~4.4]{PRW1} there exists a quadratic form $p$ on $\rr {2d + N}$ that defines $\lambda$,  
and this $p$ defines \eqref{schwartzkernel1}. 
The factor in front of the integral  \eqref{schwartzkernel1} is designed to make the oscillatory integral independent of the quadratic form $p$ on $\rr {2d + N}$, including possible changes of dimension $N$ as discussed after \cite[Proposition~4.2]{PRW1}, as long as $p$ defines $\lambda$ by means of \cite[Eq.~(4.8)]{PRW1} with $x \in \cc d$ replaced by $(x,y) \in \cc {2d}$. 

It is shown in \cite[p.~444]{Hormander2} that the kernel $K_{e^{-2 i t F}}$ is uniquely determined by the Lagrangian $\lambda$, apart from a sign ambiguity which is not essential for our purposes. 
For brevity we denote $\cK_{e^{-2 i t F}} = \cK_t$ for $t \geqs 0$. 

By \cite[p.~446]{Hormander2} the $L^2$ adjoint of $\cK_t$, 
defined by 
\begin{equation}\label{eq:L2adjoint}
(\cK_t f,g) = (f, \cK_t^* g), \quad f,g \in L^2(\rr d), 
\end{equation}
is $\cK_t^* = \cK_T$ where 
\begin{equation*}
T = \overline{ (e^{-2i t F})^{-1}} 
=  \overline{ e^{2i t F}} 
= e^{- 2 i t \overline F}. 
\end{equation*}
Thus the adjoint $\cK_t^*$ is an operator of the same type as $\cK_t$. 
It is obtained from the latter by conjugation of the matrix $F$, i.e. 
$\cK_t^* = \cK_{e^{- 2 i t F}}^* = \cK_{e^{- 2 i t \overline F}}$.

%%%%%%%%%%%%%%%%%%%%%%%%%%%%%%%%%%
\section{The propagator, multiplication and differential operators}
\label{sec:propagator}
%%%%%%%%%%%%%%%%%%%%%%%%%%%%%%%%%%

The following lemma is an important tool for our results. 
It can be seen as a commutator relation for the propagator $\cK_t$
and $x^\alpha D^\beta$ operators, and particularly the limit behavior as $t \to 0^+$.  

\begin{lem}\label{lem:operatorcommutator}
If $\alpha,\beta \in \nn d$ then
\begin{equation}\label{eq:multdiffcommut}
x^\alpha D^\beta \cK_t 
= \sum_{|\gamma + \kappa| \leqs |\alpha+\beta|} C_{\gamma,\kappa} (t) \, \cK_t \, x^\gamma D^\kappa
\end{equation}
where $[0,\infty) \ni t \mapsto C_{\gamma,\kappa} (t)$ are continuous functions that satisfy
\begin{equation}\label{eq:tlimits}
\begin{aligned}
\lim_{t \to 0^+} C_{\alpha,\beta} (t) & = 1, \\ 
\lim_{t \to 0^+} C_{\gamma,\kappa} (t) & = 0, \quad (\gamma,\kappa) \neq (\alpha,\beta).
\end{aligned}
\end{equation}
\end{lem}

\begin{proof}
Let $(x_0,\xi_0) \in T^* \cc d$ and set $(y_0(t) ,\eta_0 (t)) = e^{2 i t F} (x_0,\xi_0) \in T^* \cc d$. 
By the proof of \cite[Proposition~5.8]{Hormander2} we have
\begin{equation}\label{eq:commut1order}
\left( \la D_x, x_0 \ra - \la x,\xi_0 \ra  \right) \cK_t = \cK_t \left( \la D_x, y_0(t) \ra - \la x,\eta_0 (t) \ra  \right). 
\end{equation}

We first prove \eqref{eq:multdiffcommut} and \eqref{eq:tlimits} when $\alpha = 0$ and $\beta \in \nn d$ using induction.
Let $1 \leqs j \leqs d$ and set $x_0 = e_j$ and $\xi_0 = 0$. Then 
\begin{equation}\label{eq:limitspacecoord}
\lim_{t \to 0^+} ( y_0(t),  \eta_0 (t) ) = ( e_j, 0),
\end{equation}
so \eqref{eq:commut1order} proves \eqref{eq:multdiffcommut} and \eqref{eq:tlimits} when $|\beta| = 1$.
Suppose \eqref{eq:multdiffcommut} and \eqref{eq:tlimits} hold when $\alpha = 0$ and $|\beta| = n \geqs 1$. 
Using \eqref{eq:commut1order} we have for $1 \leqs j \leqs d$, $x_0 = e_j$ and $\xi_0 = 0$
\begin{equation*}
D^{e_j +\beta} \cK_t 
= \sum_{|\gamma + \kappa| \leqs n} C_{\gamma,\kappa} (t) \cK_t \, \left( \la D_x, y_0(t) \ra - \la x,\eta_0 (t) \ra  \right) x^\gamma D^\kappa
\end{equation*}
where $\lim_{t \to 0^+} C_{0,\beta} (t) = 1$ and $\lim_{t \to 0^+} C_{\gamma,\kappa} (t) = 0$ when $(\gamma,\kappa) \neq (0,\beta)$. 
Again using \eqref{eq:limitspacecoord} we obtain \eqref{eq:multdiffcommut} and \eqref{eq:tlimits} for $\alpha = 0$ and $|\beta| = n+1$, which constitutes the induction step. Thus the claim \eqref{eq:multdiffcommut} and \eqref{eq:tlimits} is true for $\alpha = 0$ and any $\beta \in \nn d$. 

Next let $1 \leqs j \leqs d$ and set $x_0 = 0$ and $\xi_0 = -e_j$. Then 
\begin{equation}\label{eq:tlimit2}
\lim_{t \to 0^+} ( y_0(t),  \eta_0 (t) ) = ( 0, - e_j).
\end{equation}
By combining what we have shown with \eqref{eq:commut1order} we have for $\beta \in \nn d$
\begin{equation*}
x_j D^{\beta} \cK_t 
= \sum_{|\gamma + \kappa| \leqs |\beta|} C_{\gamma,\kappa} (t) \cK_t \, \left( \la D_x, y_0(t) \ra - \la x,\eta_0 (t) \ra  \right) x^\gamma D^\kappa
\end{equation*}
where $\lim_{t \to 0^+} C_{0,\beta} (t) = 1$ and $\lim_{t \to 0^+} C_{\gamma,\kappa} (t) = 0$ when $(\gamma,\kappa) \neq (0,\beta)$. 
Invoking \eqref{eq:tlimit2} proves the claims \eqref{eq:multdiffcommut} and \eqref{eq:tlimits} for $|\alpha| = 1$ and $\beta \in \nn d$. 
The generalization to $\alpha \in \nn d$ arbitrary follows again by induction. 
\end{proof}

In the next result we use the concept of a bounded set in $\cS(\rr d)$. 
A subset $B \subseteq \cS(\rr d)$ is bounded provided each seminorm is uniformly bounded. 
Using the system of seminorms \eqref{eq:seminormsS1} this can be expressed as
\begin{equation}\label{eq:boundedsetS1}
\sup_{\fy \in B} \| \fy \|_n = C_n < \infty \quad \forall n \in \no. 
\end{equation}

We prove a few preparatory results that are needed in Section \ref{sec:strongcontmodulation},
where we show that the propagator $\cK_t$ is a strongly continuous semigroup on $M_s^2(\rr d)$ for all $s \in \ro$. 

\begin{lem}\label{lem:boundedsetinvarianceS}
If $B \subseteq \cS(\rr d)$ is bounded and $\gamma, \kappa \in \nn d$ then $\{ x^\gamma D^\kappa \fy, \ \fy \in B \} \subseteq \cS(\rr d)$ is also bounded. 
\end{lem}

\begin{proof}
We use the seminorms \eqref{eq:seminormsS1} so we assume that \eqref{eq:boundedsetS1} is valid. 
For $\alpha, \beta \in \nn d$ we have
\begin{align*}
\left| x^\alpha D ^\beta \left( x^\gamma D^\kappa \fy (x) \right) \right|
& = \left| \sum_{\sigma \leqs \min( \beta,\gamma)} \binom{\beta}{\sigma} \frac{\gamma! \, i^{-|\sigma|}}{(\gamma-\sigma)!} x^{\alpha + \gamma - \sigma} D^{\kappa+\beta-\sigma} \fy (x) \right| \\
& \leqs |\gamma|! \sum_{\sigma \leqs \min( \beta,\gamma)} \binom{\beta}{\sigma}  \left| x^{\alpha + \gamma - \sigma} D^{\kappa+\beta-\sigma} \fy (x) \right|
\end{align*}
which gives for any $n \in \no$
\begin{align*}
\| x^\gamma D^\kappa \fy \|_n
& = \max_{|\alpha+\beta| \leqs n} \sup_{x \in \rr d} \left| x^\alpha D ^\beta \left( x^\gamma D^\kappa \fy (x) \right) \right| \\
& \leqs |\gamma|! \max_{|\alpha+\beta| \leqs n} \sum_{\sigma \leqs \min( \beta,\gamma)} \binom{\beta}{\sigma} 
\sup_{x \in \rr d}  \left| x^{\alpha + \gamma - \sigma} D^{\kappa+\beta-\sigma} \fy (x) \right| \\
& \leqs |\gamma|! \, \| \fy \|_{n + |\gamma+\kappa|} \max_{|\alpha+\beta| \leqs n} \sum_{\sigma \leqs \min( \beta,\gamma)} \binom{\beta}{\sigma} \\
& \leqs |\gamma|!  \, C_{n + |\gamma+\kappa|} \, \max_{|\alpha+\beta| \leqs n} 2^{|\beta|} \\
& \leqs |\gamma|! \, 2^n \, C_{n + |\gamma+\kappa|}, \quad \fy \in B.  
\end{align*}
\end{proof}

\begin{lem}\label{lem:boundedsetcoverS}
If $B \subseteq \cS(\rr d)$ is bounded and $\ep > 0$
then there exists $K \in \no$ and $\fy_j \in \cS(\rr d)$ for $1 \leqs j \leqs K$ such that 
\begin{equation*}
B \subseteq \bigcup_{j=1}^K B_\ep (\fy_j)
 \end{equation*}
where the open balls $B_\ep (\fy_j) \subseteq L^2(\rr d)$ refer to the $L^2$ norm. 
\end{lem}

\begin{proof}
We use the identification \eqref{eq:Ssequence} of $\cS(\rr d)$ as a projective limit of sequence spaces for Hermite series expansions. 
Then $L^2(\rr d)$ corresponds to $\ell^2(\nn d)$. 
We work on the side of the sequences $c = (c_\alpha)_{\alpha \in \nn d}$. 
Since $B \subseteq \cS (\rr d)$ is bounded there exists for each $r>0$ a bound $C_r > 0$ such that 
\begin{equation*}
\| c \|^2_{\ell_r^2} = \sum_{\alpha \in \nn d} |c_\alpha|^2 \eabs{ \alpha}^{2 r}
\leqs C_r^2, \quad c \in B. 
\end{equation*}
For $r=1$ and $N \in \no$ this gives
\begin{align*}
\sum_{\alpha \in \nn d, \, |\alpha| > N} |c_\alpha|^2 
& = \sum_{\alpha \in \nn d, \, |\alpha| > N} |c_\alpha|^2 \eabs{ \alpha}^{2 - 2} \\
& \leqs \eabs{N}^{- 2 } \sum_{\alpha \in \nn d} |c_\alpha|^2 \eabs{ \alpha}^{2}
\leqs C_1^2 \eabs{N}^{- 2}, \quad c \in B. 
\end{align*}
If we pick $N > 0$ sufficiently large we thus have 
\begin{equation}\label{eq:BlargeindicesS}
\sup_{c \in B} \sum_{\alpha \in \nn d, \, |\alpha| > N} |c_\alpha|^2 < \frac{\ep^2}{2}. 
\end{equation}

On the other hand we have 
\begin{equation*}
B_{(N)} := \left\{ \{ c_\alpha\}_{|\alpha| \leqs N} : \ \{ c_\alpha \}_{\alpha \in \nn d} \in B \right\} \subseteq \cc M
\end{equation*}
for some $M \in \no$, and 
\begin{equation*}
\sum_{|\alpha| \leqs N} |c_\alpha|^2 
\leqs \sum_{\alpha \in \nn d} |c_\alpha|^2 \eabs{ \alpha}^{2}
\leqs C_1^2, \quad c \in B,  
\end{equation*}
so $B_{(N)} \subseteq \overline{B}_{C_1} \subseteq \cc M$ where $B_{C_1}$ denotes the open ball in $\cc M$, considered as a Hilbert space, with radius $C_1 > 0$. 
By the compactness of its closure $\overline{B}_{C_1} \subseteq \cc M$ there exist $\{ c_j \}_{j=1}^K \subseteq \cc M$ such that 
\begin{equation}\label{eq:mindistS}
\min_{1 \leqs j \leqs K} \| c - c_j \|_{\ell_M^2}^2 < \frac{\ep^2}{2}, \quad c \in B_{(N)}.  
\end{equation}
We extend $c_j$ to elements in $\ell^2(\nn d)$ by zero-padding: 
\begin{equation*}
c_{j,\alpha} = 0, \quad |\alpha| > N, \quad 1 \leqs j \leqs K. 
\end{equation*}
Combining \eqref{eq:BlargeindicesS} and \eqref{eq:mindistS} gives 
\begin{equation*}
\min_{1 \leqs j \leqs K} \| c - c_j \|_{\ell^2 (\nn d)}^2 
= \min_{1 \leqs j \leqs K}  \sum_{|\alpha| \leqs N} |c_\alpha - c_{j,\alpha}|^2 +  \sum_{|\alpha| > N} |c_\alpha|^2
< \ep^2, \quad c \in B. 
\end{equation*}
Thus
\begin{equation*}
B \subseteq \bigcup_{j=1}^K B_\ep (c_j). 
\end{equation*}
\end{proof}

\begin{lem}\label{lem:strongcont}
If $B \subseteq \cS(\rr d)$ is bounded and $\alpha,\beta \in \nn d$ then
\begin{equation*}
\lim_{t \to 0^+} \sup_{\fy \in B} \| x^\alpha D^\beta (\cK_t -I) \fy \|_{L^2} = 0. 
\end{equation*}
\end{lem}

\begin{proof}
From Lemma \ref{lem:operatorcommutator} we obtain for $\fy \in \cS$
\begin{align*}
x^\alpha D^\beta (\cK_t -I) \fy 
& = C_{\alpha,\beta} (t) (\cK_t -I) x^\alpha D^\beta \fy + ( C_{\alpha,\beta} (t) - 1) x^\alpha D^\beta \fy \\
& \qquad + \sum_{\stackrel{|\gamma + \kappa| \leqs |\alpha+\beta|}{(\gamma,\kappa) \neq (\alpha,\beta)}} C_{\gamma,\kappa} (t) \cK_t \, x^\gamma D^\kappa \fy
\end{align*}
where \eqref{eq:tlimits} holds. 
The contraction property of $\cK_t$ acting on $L^2$ yields for $0 < t \leqs 1$
\begin{equation}\label{eq:smalltimeestimateS}
\begin{aligned}
\| x^\alpha D^\beta (\cK_t -I) \fy \|_{L^2}
& \leqs C \| (\cK_t -I) x^\alpha D^\beta \fy \|_{L^2} + |C_{\alpha,\beta} (t) - 1| \, \| x^\alpha D^\beta \fy \|_{L^2} \\
& \qquad + \sum_{\stackrel{|\gamma + \kappa| \leqs |\alpha+\beta|}{(\gamma,\kappa) \neq (\alpha,\beta)}} |C_{\gamma,\kappa} (t)| \, \| x^\gamma D^\kappa \fy \|_{L^2}
\end{aligned}
\end{equation}
where $C>0$.

Let $\ep > 0$. 
By Lemmas \ref{lem:boundedsetinvarianceS} and \ref{lem:boundedsetcoverS} there exists $K \in \no$ and $\fy_j \in \cS(\rr d)$, $1 \leqs j \leqs K$, such that 
\begin{equation*}
\min_{1 \leqs j \leqs K} \| x^\alpha D^\beta \fy - \fy_j \|_{L^2} < \frac{\ep}{8 C}, \quad \fy \in B. 
\end{equation*}

Next we use two properties of $\cK_t$ acting on $L^2$: the contraction property and the strong continuity. 
This gives for $0 < t \leqs \delta$
\begin{equation}\label{eq:firsttermS}
\begin{aligned}
\| (\cK_t -I) x^\alpha D^\beta \fy \|_{L^2} 
& = \min_{1 \leqs j \leqs K} \| (\cK_t -I) (x^\alpha D^\beta \fy - \fy_j + \fy_j) \|_{L^2} \\
& \leqs \min_{1 \leqs j \leqs K} \left( 2 \| x^\alpha D^\beta \fy - \fy_j\|_{L^2} + \| (\cK_t -I) \fy_j \|_{L^2} \right) \\
& \leqs \frac{\ep}{4 C} + \frac{\ep}{4 C} =  \frac{\ep}{2 C}, \quad \fy \in B, 
\end{aligned}
\end{equation}
provided $\delta > 0$ is sufficiently small. 

In the next step we use the seminorms \eqref{eq:seminormsS1} for $\cS$ and \eqref{eq:boundedsetS1}. 
We also use
\begin{equation}\label{eq:bracketdsquare}
\eabs{x}^{2d} = ( 1 + x_1^2 + \cdots  + x_d^2)^d = \sum_{|\sigma| \leqs d} C_\sigma x^{2 \sigma}
\end{equation}
where $C_\sigma > 0$ are constants. 
Thus we obtain for $|\gamma + \kappa| \leqs |\alpha+\beta|$
\begin{equation}\label{eq:secondtermS}
\| x^\gamma D^\kappa \fy \|_{L^2}^2
= \sum_{|\sigma| \leqs d} C_\sigma \int_{\rr d} \eabs{x}^{- 2d} | x^{\sigma+\gamma} D^\kappa \fy  (x)|^2 \, \dd x 
\leqs D_1^2 \, C_{|\alpha+\beta|+d}^2 , \quad \fy \in B, 
\end{equation}
for some $D_1 > 0$. 

Finally we insert \eqref{eq:firsttermS} and \eqref{eq:secondtermS} into \eqref{eq:smalltimeestimateS}. 
We obtain then for $0 < t \leqs \delta$, again after possibly decreasing $\delta > 0$, 
\begin{align*}
& \| x^\alpha D^\beta (\cK_t -I) \fy \|_{L^2} \\
& \leqs C \| (\cK_t -I) x^\alpha D^\beta \fy \|_{L^2} + D_1 C_{|\alpha+\beta|+d} 
\, \left( |C_{\alpha,\beta} (t) - 1| 
+ \sum_{\stackrel{|\gamma + \kappa| \leqs |\alpha+\beta|}{(\gamma,\kappa) \neq (\alpha,\beta)}} |C_{\gamma,\kappa} (t)| \, 
 \right)  \\
& \leqs \frac{\ep}{2} + \frac{\ep}{2} = \ep, \quad \fy \in B. 
\end{align*}
Since $\ep > 0$ is arbitrary this proves the claim. 
\end{proof}

Lemma \ref{lem:operatorcommutator} is useful in order to understand the behavior of the propagator $\cK_t$ as $t \to 0^+$, witness Lemma \ref{lem:strongcont}. 
We will prove more results in this direction further on, see Theorems \ref{thm:semigroupMs} and \ref{thm:strongcontGS}.

%%%%%%%%%%%%%%%%%%%%%%%%%%%%%%%%%%%%%%%%%
\section{Strong continuity on Hilbert modulation spaces and tempered distributions}
\label{sec:strongcontmodulation}
%%%%%%%%%%%%%%%%%%%%%%%%%%%%%%%%%%%%%%%%%

In this section we prove that $\cK_t$ is a strongly continuous semigroup in several subspaces of the tempered distributions: $M_s^2(\rr d)$ for any $s \in \ro$, 
the Schwartz space $\cS(\rr d)$, and $\cS'(\rr d)$ equipped with either the weak$^*$ or the strong topology. 
In the case of $\cS'(\rr d)$ equipped with the strong topology, we show that the semigroup is locally equicontinuous. 

We need the following tool in the proof of Theorem \ref{thm:semigroupMs}. 

\begin{lem}\label{lem:boundedpropagator}
Let $s \in \ro$ and $T>0$. 
The propagator $\cK_t$ is bounded on $M_s^2(\rr d)$ uniformly over $0 \leqs t \leqs T$. 
\end{lem}

\begin{proof}
By \cite[Theorem~4.5]{Feichtinger1} (cf. \cite[Proposition~1.2]{Holst1}) the modulation spaces are closed under complex interpolation of Banach spaces \cite{Bergh1}. 
We may thus assume that $s=k \in \zo$. 
Suppose $k \geqs 0$. By \cite[Theorem~2.1.12]{Nicola1} 
\begin{equation}\label{eq:Qknorm}
\| u \| = \sum_{|\alpha + \beta| \leqs k} \| x^\alpha D^\beta u \|_{L^2}
\end{equation}
is a norm on $M_k^2(\rr d)$ that is equivalent to \eqref{eq:Qsnorm}. 

From Lemma \ref{lem:operatorcommutator} and the contraction property of $\cK_t$ we obtain
\begin{align*}
\| \cK_t u \|_{M_k^2}
\asymp \sum_{|\alpha + \beta| \leqs k} \| x^\alpha D^\beta \cK_t u \|_{L^2}
& \leqs \sum_{|\alpha + \beta| \leqs k} \, \sum_{|\gamma + \kappa| \leqs |\alpha+\beta|} | C_{\gamma,\kappa} (t) | \, \|  x^\gamma D^\kappa u \|_{L^2} \\
& \lesssim \sum_{|\alpha + \beta| \leqs k}  \|  x^\alpha D^\beta u \|_{L^2} \asymp \| u \|_{M_k^2}, 
\end{align*}
in the last inequality using the consequence of Lemma \ref{lem:operatorcommutator} that the functions $C_{\gamma,\kappa}$ are continuous and therefore uniformly bounded with respect to $t \in [0,T]$.  
This proves the lemma when $k \geqs 0$. 

If $k < 0$ we use duality. In fact the dual of $M_k^2$ can be identified with $M_{-k}^2$ with respect to an extension of the $L^2$ inner product \cite{Feichtinger1}, \cite[Theorem~11.3.6]{Grochenig1}. 
We also use the expression of the adjoint of $\cK_t$ as $\cK_t^* = \cK_{e^{- 2 i t \overline{F}}}$, cf. \eqref{eq:L2adjoint}. 
By the result above we have 
\begin{equation*}
\| \cK_{e^{- 2 i t \overline{F}}} u \|_{M_{-k}^2} \lesssim  \| u \|_{M_{-k}^2}, \quad 0 \leqs t \leqs T,
\end{equation*}
which gives
\begin{align*}
\| \cK_t u \|_{M_k^2}
& = \sup_{\| g \|_{M_{-k}^2} \leqs 1} \left| \left(  \cK_t u, g \right) \right|
= \sup_{\| g \|_{M_{-k}^2} \leqs 1} \left| \left(  u, \cK_{e^{- 2 i t \overline{F}}} g \right)\right| \\
& \leqs \| u \|_{M_k^2} \ \sup_{\| g \|_{M_{-k}^2} \leqs 1} \| \cK_{e^{- 2 i t \overline{F}}} g \|_{M_{-k}^2}  
\lesssim  \| u \|_{M_k^2}, \quad 0 \leqs t \leqs T. 
\end{align*}
\end{proof}

\begin{thm}\label{thm:semigroupMs}
Let $s \in \ro$. 
The propagator $\cK_t = e^{- t q^w (x,D)}$ is for $t \geqs 0$ a strongly continuous semigroup on $M_s^2(\rr d)$. 
\end{thm}

\begin{proof}
By Lemma \ref{lem:boundedpropagator} the operators $\cK_t$ are bounded on $M_s^2$, uniformly over $t \in [0,T]$ for any $T>0$. 
Pick $k \in \no$ such that $k \geqs s$. 
For any $\fy \in \cS(\rr d)$ we obtain from \eqref{eq:shubinsobolevnested}, using the norm \eqref{eq:Qknorm} on $M_k^2$, and  Lemma \ref{lem:strongcont}
\begin{equation*}
\| (\cK_t -I) \fy \|_{M_s^2}
\lesssim \| (\cK_t -I) \fy \|_{M_k^2}
\asymp \sum_{|\alpha + \beta| \leqs k} \| x^\alpha D^\beta (\cK_t -I) \fy \|_{L^2} \longrightarrow 0, \quad t \to 0^+. 
\end{equation*}
Since $\cS \subseteq M_s^2$ is dense \cite[Proposition~11.3.4]{Grochenig1}, we may combine this find, Lemma \ref{lem:boundedpropagator}
and \cite[Proposition~I.5.3]{Engel1}. 
The conclusion of the latter result is then the strong continuity of $\cK_t$ on $M_s^2(\rr d)$. 

Finally we consider the semigroup property.  
If $s \geqs 0$ then $M_s^2 \subseteq L^2$. 
Thus $\cK_0 = I$ and $\cK_{t_1 + t_2} = \cK_{t_1} \cK_{t_2}$ hold on $M_s^2(\rr d)$ due to the corresponding properties on $L^2$. 
If $s < 0$ then let $u \in M_s^2(\rr d)$ and let $t_1, t_2 \geqs 0$. 
From the extension of \eqref{eq:L2adjoint} to the duality on $M_{-s}^2 \times M_{s}^2$ we have 
for $\fy \in \cS(\rr d)$
\begin{equation*}
(  (\cK_{t_1 + t_2} - \cK_{t_1} \cK_{t_2}) u, \fy) = (u, (\cK_{t_1 + t_2}^* - \cK_{t_2}^* \cK_{t_1}^* ) \fy ) = 0
\end{equation*}
due to the semigroup property $\cK_{t_1 + t_2}^* = \cK_{t_2}^* \cK_{t_1}^*$ when the action refers to $L^2$. 
This proves the semigroup property $\cK_{t_1 + t_2} = \cK_{t_1} \cK_{t_2}$ for action on $M_s^2(\rr d)$, 
and likewise $\cK_0 = I$ on $M_s^2(\rr d)$. 
\end{proof}

\begin{cor}\label{cor:propagatorschwartz}
The propagator $\cK_t$ is for $t \geqs 0$ a locally equicontinuous strongly continuous semigroup on $\cS(\rr d)$. 
\end{cor}

\begin{proof}
We use the seminorms \eqref{eq:seminormsS2} on $\cS(\rr d)$. 
The continuity of $\cK_t$ on $\cS$ follows from Lemma \ref{lem:boundedpropagator}, 
as well as the local equicontinuity. 
The strong continuity is a consequence of the proof of Theorem \ref{thm:semigroupMs}. 
Finally the semigroup property $\cK_{t_1 + t_2} = \cK_{t_1} \cK_{t_2}$ for $t_1, t_2 \geqs 0$, and $\cK_0 = I$, 
are immediate consequences of the corresponding properties for the semigroup acting on $L^2$. 
\end{proof}

The generator of the semigroup $\cK_t$ acting on $M_s^2(\rr d)$ 
according to Theorem \ref{thm:semigroupMs} is 
\begin{equation}\label{eq:generatorMs}
A_s f = \lim_{h \to 0^+} 
h^{-1} \left( \cK_h - I \right) f
\end{equation}
for all $f \in M_s^2(\rr d)$ such that the right-hand side limit exists in $M_s^2(\rr d)$ \cite{Pazy1}.
The linear space of all such $f \in M_s^2(\rr d)$ is the domain of $A_s$ denoted $D(A_s) \subseteq M_s^2(\rr d)$. 
For each $s \in \ro$ the operator $A_s$ equipped with the domain $D(A_s)$ is an unbounded linear operator in $M_s^2(\rr d)$. 
The domain $D(A_s)$ is dense in $M_s^2(\rr d)$ and the operator $A_s$ is closed \cite[Corollary~I.2.5]{Pazy1}. 

It follows from \eqref{eq:shubinsobolevnested} that $D(A_{s_2}) \subseteq D(A_{s_1})$ if $s_1 \leqs s_2$ and $A_{s_2} f = A_{s_1} f$ if $f \in D(A_{s_2})$. 
Thus we have for $0 \leqs s_1 \leqs s_2$
\begin{equation}\label{eq:operatorsnested}
A_{s_2} \subseteq A_{s_1} \subseteq - q^w(x,D) \subseteq A_{-s_1} \subseteq A_{-s_2}
\end{equation}
where $- q^w(x,D) = A_0$ denotes the generator of the semigroup $\cK_t$ on $L^2$. 

According to Corollary \ref{cor:propagatorschwartz} the propagator $\cK_t$ is also a locally equicontinuous strongly continuous semigroup on $\cS$. 
The generator of the semigroup $\cK_t$ acting on $\cS$ is 
\begin{equation}\label{eq:generatorS}
A f = \lim_{h \to 0^+} h^{-1} \left( \cK_h - I \right)f
\end{equation}
for all $f \in \cS$ such that the limit is well defined in $\cS$. 
The space of such $f$ is the domain denoted $D(A) \subseteq \cS$. 
According to \cite[Propositions~1.3 and 1.4]{Komura1} $A$ is a closed linear operator and $D(A) \subseteq \cS$ is dense (cf. Remark \ref{rem:SCLECauchy}).  

Let $f \in D(A)$ and let $s \geqs 0$. Then \eqref{eq:generatorS} converges in $\cS$ and therefore also in $M_s^2$, to the same element in $M_s^2$.
Thus $f \in D(A_s)$, so this means that $D(A) \subseteq D(A_s)$ and $A \subseteq A_s$.  
In particular for $s=0$ we have $A f = - q^w(x,D) f$ if $f \in D(A) \subseteq \cS$. 
By \cite[p.~178]{Shubin1} $q^w(x,D)$ is continuous on $\cS$. 
Since $A$ is closed and $D(A) \subseteq \cS$ is dense we must have $D(A) = \cS$. 
Combined with \eqref{eq:operatorsnested} his yields
\begin{equation*}
\cS = D(A) \subseteq \bigcap_{s \in \ro} D(A_s). 
\end{equation*}
If $f \in \bigcap_{s \in \ro} D(A_s)$ then \eqref{eq:generatorS} converges in $\cS$ 
so $f \in D(A) = \cS$ and we can strengthen the inclusion into
\begin{equation}\label{eq:equalityS}
\cS = \bigcap_{s \in \ro} D(A_s). 
\end{equation}

It follows from above that $A$ is continuous on $\cS$. 
We can thus extend $A$ uniquely to $\cS'$, using its formal $L^2$ adjoint $A^* = - \overline{q}^w(x,D)$ acting on $\cS$, by 
\begin{equation}\label{eq:Aextension}
(A u,\fy) = (u, A^* \fy), \quad u \in \cS'(\rr d), \quad \fy \in \cS(\rr d).  
\end{equation}
The extension is continuous on $\cS'$ equipped with its weak$^*$ topology. 

\begin{lem}\label{lem:domaininclusion}
For each $s \in \ro$ we have $M_{s+2}^2(\rr d) \subseteq D(A_s)$. 
\end{lem}

\begin{proof}
Since $q \in \Gamma^2$ we have by \eqref{eq:shubincontMs} for any $s \in \ro$
\begin{equation}\label{eq:contshubin}
\| q^w(x,D) f \|_{M_s^2} \lesssim \| f \|_{M_{s+2}^2}, \quad f \in \cS. 
\end{equation}
Let $f \in M_{s+2}^2 (\rr d)$. Since $\cS \subseteq M_{s+2}^2$ is a dense subspace \cite[Proposition~11.3.4]{Grochenig1} there exists a sequence $(f_n)_{n \geqs 1} \subseteq \cS$ such that $f_n \to f$ in $M_{s+2}^2$ as $n \to \infty$. 
By \eqref{eq:shubinsobolevnested} this implies that 
\begin{equation}\label{eq:limitMs}
f_n \to f \quad \mbox{in} \quad M_{s}^2 \quad \mbox{as} \quad n \to \infty. 
\end{equation}
From \eqref{eq:equalityS} we know that $\cS \subseteq D(A_s) \bigcap D(q^w(x,D) )$ and hence using \eqref{eq:contshubin} we obtain
\begin{equation*}
\| A_s (f_n - f_m) \|_{M_s^2} = \| q^w(x,D) (f_n - f_m) \|_{M_s^2} 
\lesssim \| f_n - f_m \|_{M_{s+2}^2} 
\end{equation*}
for $n,m \geqs 1$. 
Thus $( A_s f_n )_{n \geqs 1}$ is a Cauchy sequence in $M_s^2$,    
which converges to an element $g \in M_s^2$.  
If we combine $A_s f_n \to g$ in $M_s^2$ as $n \to \infty$ with \eqref{eq:limitMs} and the fact that $A_s$ is closed, 
we may conclude that $f \in D(A_s)$ and $A_s f = g$.
Hence $M_{s+2}^2 (\rr d) \subseteq D(A_s)$. 
\end{proof}

When we consider the equation \eqref{eq:CP} in $M_s^2(\rr d)$, we identify $A_s = - q^w(x,D)$. 

\begin{cor}\label{cor:uniquenessCP}
Let $s \in \ro$ and consider the Cauchy problem \eqref{eq:CP} in $M_s^2(\rr d)$. 
If $u_0 \in M_{s+2}^2(\rr d)$ then $\cK_t u_0$ is the unique solution in $C^1 ( [0,\infty), M_s^2)$.  
\end{cor}

\begin{proof}
The claim is a consequence of Lemma \ref{lem:domaininclusion}, \cite[Theorem~I.2.4]{Pazy1} and \cite[pp.~483--84]{Kato1}. 
\end{proof}

Finally we obtain from \eqref{eq:shubinsobolevspaces} the following consequence. 

\begin{cor}\label{cor:uniquenessCP2}
The Cauchy problem \eqref{eq:CP} has the solution $\cK_t u_0$ for any $u_0 \in \cS' (\rr d)$. 
It is unique in the sense of Corollary \ref{cor:uniquenessCP}. 
\end{cor}

A version of Corollary \ref{cor:uniquenessCP2} with additional information can be obtained in another fashion, as follows. 

By Corollary \ref{cor:propagatorschwartz} we may for fixed $t \geqs 0$ extend $\cK_t$ from domain $\cS(\rr d)$ to $\cS'(\rr d)$ uniquely by defining
\begin{equation}\label{eq:dualsemigroup}
(\cK_t u, \fy) = (u, \cK_t^* \fy) = (u, \cK_{e^{- 2 i t \overline F}} \fy), \quad u \in \cS'(\rr d), \quad \fy \in \cS(\rr d), 
\end{equation}
since $\cK_t^* \fy \in \cS$, cf. \eqref{eq:L2adjoint}. 
Then $\cK_{t_1 + t_2} = \cK_{t_1} \cK_{t_2}$ for $t_1, t_2 \geqs 0$ and $\cK_0 = I$
for the action on $\cS'$ follows as in the proof of Theorem \ref{thm:semigroupMs}. 

Denote by $\cS_{\rm w}'$ the space $\cS'$ equipped with its weak$^*$ topology, with seminorms $\cS' \ni u \mapsto |(u,\fy)|$ for all $\fy \in \cS$. 
From Corollary \ref{cor:propagatorschwartz} it follows that $\cK_t: \cS_{\rm w}' \to \cS_{\rm w}'$ is continuous for each $t \geqs 0$. 
Let $u \in \cS'(\rr d)$. For some $s \geqs 0$ we have for $\fy \in \cS$
\begin{equation*}
|((\cK_t -I)u, \fy)| = |(u, (\cK_t^*-I) \fy)|
\lesssim \| (\cK_t^*-I) \fy \|_{M_s^2}. 
\end{equation*}
The right-hand side approaches zero as $t \to 0^+$ according to Theorem \ref{thm:semigroupMs}. 
We may conclude that $\cK_t$ is a strongly continuous semigroup on $\cS_{\rm w}'$. 

The modulus of the right-hand side of \eqref{eq:dualsemigroup} equals 
$|(u, \cK_t^* \fy)|$. 
For $t$ in the interval $0 \leqs t \leqs T < \infty$ with $T > 0$ given, this is an indexed family of seminorms of $u \in \cS_{\rm w}'$, 
but we cannot estimate $\{ |(u, \cK_t^* \fy)| \}_{0 \leqs t \leqs T}$ by a \emph{single} seminorm. 
Thus we cannot show that the semigroup $\cK_t$ is locally equicontinuous on $\cS_{\rm w}'$. 
For that purpose we need to equip $\cS'$ with another topology. 

The space $\cS_{\rm str}'$ denotes $\cS'$ equipped with its strong topology \cite{Reed1}, with seminorms 
\begin{equation*}
\cS' \ni u \mapsto \sup_{\fy \in B} |(u,\fy)|
\end{equation*}
for each bounded set $B \subseteq \cS$. 
Expressed with the seminorms \eqref{eq:seminormsS2} 
a bounded set satisfies
\begin{equation*}
\sup_{\fy \in B} \| \fy \|_{M_s^2} = C_s < \infty, \quad \forall s \geqs 0. 
\end{equation*}

If $B \subseteq \cS$ is bounded and $0 \leqs t \leqs T$ then
\begin{equation*}
\sup_{\fy \in B} |(\cK_t u, \fy)| = \sup_{\fy \in B} |(u, \cK_t^* \fy)| 
\leqs \sup_{\fy \in B, \ 0 \leqs t \leqs T} |(u, \cK_t^* \fy)| , \quad u \in \cS'. 
\end{equation*}
By Lemma \ref{lem:boundedpropagator} $\{ \cK_t^* B, \ 0 \leqs t \leqs T\} \subseteq \cS$ is a bounded set. 
This shows that $\cK_t$ is continuous on $\cS_{\rm str}'$ for each $t \geqs 0$, and $\{ \cK_t \}_{t \geqs 0}$ is a locally equicontinuous semigroup on $\cS_{\rm str}'$. 
It is also a strongly continuous semigroup on $\cS_{\rm str}'$. 
In fact let $u \in \cS'(\rr d)$ and let $B \subseteq \cS$ be bounded. 
We have for some $k \in \no$ using \eqref{eq:Qknorm} and Lemma \ref{lem:strongcont} 
\begin{align*}
\sup_{\fy \in B} |((\cK_t -I)u, \fy)| 
& = \sup_{\fy \in B} |(u, (\cK_t^* - I)\fy)| \\
& \lesssim \sup_{\fy \in B} \| (\cK_t^* - I)\fy) \|_{M_k^2} \\
& \leqs \sum_{|\alpha+\beta| \leqs k} \sup_{\fy \in B}  \| x^\alpha D^\beta (\cK_t^* - I) \fy \|_{L^2} \\
& \quad \longrightarrow 0, \quad t \to 0^+. 
\end{align*}

We have proved: 

\begin{thm}\label{thm:strongcontlocequicontS}
The semigroup $\cK_t$ is: 
\begin{enumerate}[\rm (i)] 
\item strongly continuous on $\cS_{\rm w}'$, and
\item locally equicontinuous strongly continuous on $\cS_{\rm str}'$.  
\end{enumerate}
\end{thm}

The generator of the semigroup $\cK_t$ on $\cS_{\rm w}'$ is denoted
\begin{equation*}
A_{\rm w}' u = \lim_{h \to 0^+} h^{-1} \left( \cK_h - I \right) u
\end{equation*}
for all $u \in \cS_{\rm w}'$, denoted $D(A_{\rm w}') \subseteq \cS_{\rm w}'$, such that the limit is well defined in $\cS_{\rm w}'$.
The generator of the semigroup $\cK_t$ on $\cS_{\rm str}'$ is denoted $A_{\rm str}'$. 
Note that $A_{\rm str}' \subseteq A_{\rm w}'$. 
By \cite[Proposition~2.1]{Komura1} 
$A_{\rm w}' = A$ defined by \eqref{eq:Aextension} and hence $D(A_{\rm w}') = \cS'$. 

The local equicontinuity of $\cK_t$ acting on $\cS_{\rm str}'$ guarantees by \cite[Proposition~1.4]{Komura1} 
that the operator $A_{\rm str}'$ is closed. 
By \cite[Proposition~1.3]{Komura1}, the inclusion $D(A_{\rm str}') \subseteq \cS'$ is dense. 
Combining the latter two facts gives $D(A_{\rm str}') = \cS'$ and $A_{\rm str}' = A_{\rm w}' = A$. 
The generators of the two semigroups are identical. 

We denote $A' = A_{\rm str}' = A_{\rm w}'$, 
and $A$ is defined by  \eqref{eq:generatorS}. 
Extending \eqref{eq:operatorsnested} we thus have for $s_1 \leqs s_2$
\begin{equation*}
A \subseteq A_{s_2} \subseteq A_{s_1} \subseteq A'. 
\end{equation*}

\begin{rem}\label{rem:abstractdual}
There is also a more abstract motivation for some of the conclusions above, 
based on the fact that the space $\cS$ is reflexive \cite[Theorem~V.24]{Reed1}.  
Theorem \ref{thm:strongcontlocequicontS} (i) is an immediate consequence of the definition \eqref{eq:dualsemigroup}, cf. \cite[p.~262]{Komura1}. 
The reflexivity of $\cS$ entails the following consequence by \cite[Theorem 1 and its Corollary]{Komura1}.  
The semigroup $\cK_t$, considered as a strongly continuous 
semigroup on $\cS_{\rm w}'$, is automatically a strongly continuous semigroup on $\cS_{\rm str}'$, and the two semigroups have identical infinitesimal generators. 
\end{rem}

An appeal to \cite[Proposition~1.2]{Komura1} and \cite[pp.~483--84]{Kato1} gives 
a version of Corollary \ref{cor:uniquenessCP2} with a continuity statement. 
Note that the uniqueness space is larger than the solution space: 
$C^1([0,\infty),\cS_{\rm str}') \subseteq C^1([0,\infty),\cS_{\rm w}')$. 

\begin{cor}\label{cor:uniquenessCP3}
For any $u_0 \in \cS' (\rr d)$
the Cauchy problem \eqref{eq:CP} has the solution $\cK_t u_0$ in the space 
$C^1([0,\infty),\cS_{\rm str}')$. 
The solution is unique in the space $C^1([0,\infty),\cS_{\rm w}')$. 
\end{cor}

\begin{rem}\label{rem:SCLECauchy}
A strongly continuous semigroup $T_t$ in a locally convex space $X$ has the following interesting property. 
The map $[0,\infty) \ni t \mapsto T_t u_0$ is a solution to \eqref{eq:CP} (with $q^w(x,D)$ replaced by $-A$) in $C^1([0,\infty), X)$
when $u_0 \in D(A)$ where $A$ denotes the generator of the semigroup \cite[Proposition~1.2]{Komura1}. 
The proof in \cite{Komura1} uses integrals of $T_t u_0$ with respect to $t$ over finite intervals in $[0,\infty)$. 
Thanks to the strong continuity such integrals are well defined as Riemann integrals. 
Local equicontinuity is not needed to define integrals, as is done e.g. in the proof of \cite[Theorem~IX.3.1]{Yosida1}. 
The solution $T_t u_0$ is unique in $C^1([0,\infty), X)$ by the argument in \cite[pp.~483--84]{Kato1}. 

If the space $X$ is sequentially complete then the domain $D(A) \subseteq X$ is dense \cite[Proposition~1.3]{Komura1}. 
If the semigroup $T_t$ is locally equicontinuous then the generator $A$ is a closed operator \cite[Proposition~1.4]{Komura1}. 
\end{rem}

%%%%%%%%%%%%%%%%%%%%%%%%%%%%%%%%%%%
\section{Strong continuity on Gelfand--Shilov (ultradistribution) spaces}
\label{sec:GelfandShilov}
%%%%%%%%%%%%%%%%%%%%%%%%%%%%%%%%%%%

In this section we study the semigroup $\cK_t$ acting on the Gelfand--Shilov space $\Sigma_s(\rr d)$ for $s > \frac{1}{2}$ and its dual space of ultradistributions $\Sigma_s'(\rr d)$.

We need the following lemma which  
is similar to \cite[Theorem~6.1.6]{Nicola1}. 
It is basically a special case of \cite[Remark~2.1]{Langenbruch1}, but we provide an elementary proof 
in order to give a selfcontained account as a service to the reader. 

\begin{lem}\label{lem:seminorms}
If $s > \frac{1}{2}$ then
the family of seminorms 
\begin{equation}\label{eq:seminormL2}
\nm f {h} \equiv \sup_{\alpha,\beta \in \nn d} \frac {\| x^\alpha D ^\beta f \|_{L^2}}{h^{|\alpha + \beta |}(\alpha !\, \beta !)^s}
\end{equation}
for $h > 0$, is equivalent to the family $\{ \| \cdot \|_{\mathcal S_{s,h}} \}_{h > 0}$ as seminorms on $\Sigma_s(\rr d)$. 
\end{lem}

\begin{proof}
Using $(\alpha + \gamma)! \leqs 2^{|\alpha+\gamma|} \alpha! \gamma!$ (cf. \cite[Eq.~(0.3.6)]{Nicola1})
we have for $\alpha, \beta \in \nn d$ and $0 < h \leqs 1$, cf. \eqref{eq:bracketdsquare}, 
\begin{align*}
\| x^\alpha D ^\beta f \|_{L^2}
& = \| \eabs{x}^{- d} \eabs{x}^{d} x^\alpha D ^\beta f \|_{L^2} 
\lesssim \sum_{|\gamma| \leqs d} \| x^{\alpha+\gamma} D ^\beta f \|_{L^\infty} \\
& \leqs \| f \|_{\mathcal S_{s,h}}  \sum_{|\gamma| \leqs d} h^{|\alpha + \gamma + \beta |}((\alpha + \gamma)!\, \beta !)^s \\
& \lesssim \| f \|_{\mathcal S_{s,h}} (2^s h)^{|\alpha + \beta |}  (\alpha! \, \beta !)^s. 
\end{align*}
This gives $\| f \|_{2^s h} \lesssim \| f \|_{\mathcal S_{s,h}}$,
or equivalently $\| f \|_{h} \lesssim \| f \|_{\mathcal S_{s,2^{-s} h}}$ 
for $0 < h \leqs 2^s$. 
Since $\| \cdot \|_{h_1} \leqs \| \cdot \|_{h_2}$ when $h_1 \geqs h_2 > 0$
this shows that any seminorm $\| \cdot \|_h$ with $h > 0$ can be estimated by a seminorm from $\{ \| \cdot \|_{\mathcal S_{s,h}} \}_{h > 0}$.

For an opposite estimate, 
again for $\alpha, \beta \in \nn d$ and $h > 0$ we have using Fourier's inversion formula and Plancherel's identity for $x \in \rr d$
\begin{align*}
| x^\alpha D ^\beta f (x)|
& = \left| (2 \pi)^{- \frac{d}{2} } \int_{\rr d}  \eabs{\xi}^{- 2 d} \eabs{\xi}^{2 d} \widehat{x^\alpha D ^\beta f } (\xi) e^{i \la x, \xi \ra} \, \dd \xi \right| \\
& \lesssim  \left\| \sum_{|\gamma| \leqs 2 d}  C_\gamma \, \xi^\gamma \widehat{x^\alpha D ^\beta f } (\xi) \right\|_{L^2}
\lesssim \sum_{|\gamma| \leqs 2 d}  \left\| \cF \left( D^\gamma \left( x^\alpha D ^\beta f\right) \right) \right\|_{L^2} \\
& = \sum_{|\gamma| \leqs 2 d}  \left\| D^\gamma \left( x^\alpha D ^\beta f\right) \right\|_{L^2}
\leqs \sum_{|\gamma| \leqs 2 d} \ \sum_{\kappa \leqs \min(\gamma,\alpha)} \binom{\gamma}{\kappa} 
\frac{\alpha!}{(\alpha-\kappa)!} \left\| x^{\alpha-\kappa} D ^{\beta+\gamma-\kappa} f \right\|_{L^2} \\
& \leqs 2^{|\alpha|} \sum_{|\gamma| \leqs 2 d} \ \sum_{\kappa \leqs \min(\gamma,\alpha)}  \binom{\gamma}{\kappa} 
\kappa! \left\| x^{\alpha-\kappa} D ^{\beta+\gamma-\kappa} f \right\|_{L^2}
\end{align*}
in the last step using $\alpha! = (\alpha-\kappa + \kappa)! \leqs (\alpha-\kappa)! \, \kappa! \, 2^{|\alpha|}$. 

Next we use $1 = 2 s - \delta$ where $\delta > 0$, 
and $\kappa! \geqs |\kappa|! \, d^{-|\kappa|}$ for $\kappa \in \nn d$ \cite[Eq.~(0.3.3)]{Nicola1} which gives 
\begin{equation}\label{eq:exponentialestimate}
\kappa!^{-\delta} h^{- 2 |\kappa|} 
= \left( \frac{h^{- \frac{2 |\kappa|}{\delta}}}{\kappa!} \right)^\delta
\leqs \left( \frac{ \left( d h^{- \frac{2}{\delta}} \right)^{|\kappa|} }{|\kappa|!} \right)^\delta
\leqs \exp \left( \delta d h^{- \frac{2}{\delta}} \right). 
\end{equation}
Thus for $0 < h \leqs 1$ and $x \in \rr d$
\begin{equation*}
\begin{aligned}
| x^\alpha D ^\beta f (x)|
& \leqs \| f \|_{h} \, 2^{|\alpha|} \sum_{|\gamma| \leqs 2 d}  \ \sum_{\kappa \leqs \min(\gamma,\alpha)} \binom{\gamma}{\kappa} 
\kappa!^{2s-\delta} 
h^{|\alpha + \beta+\gamma - 2\kappa|} \left( (\alpha-\kappa)! (\beta+\gamma-\kappa)! \right)^s \\
& \leqs \| f \|_{h} \, 2^{|\alpha|} h^{|\alpha + \beta|}  
\sum_{|\gamma| \leqs 2 d} \  \sum_{\kappa \leqs \min(\gamma,\alpha)} 
\binom{\gamma}{\kappa} 
\kappa!^{-\delta} h^{- 2 |\kappa|} \left( \alpha! (\beta+\gamma)! \right)^s \\
& \lesssim \| f \|_{h} \, (2h)^{|\alpha + \beta|}  \left( \alpha! \beta! \right)^s 
\sum_{|\gamma| \leqs 2 d}  \ \sum_{\kappa \leqs \min(\gamma,\alpha)}
\binom{\gamma}{\kappa} 2^{s |\beta|} \\
& \lesssim \| f \|_{h} \, (2^{1+s} h)^{|\alpha + \beta|}  \left( \alpha! \beta! \right)^s 
\end{aligned}
\end{equation*}
which gives 
$\| f \|_{\mathcal S_{s, 2^{1+s} h}} \lesssim \| f \|_{h}$,
or equivalently 
$\| f \|_{\mathcal S_{s,h}} \lesssim \| f \|_{2^{-1-s} h}$
for any $0 < h \leqs 2^{1+s}$. 
Since $\| \cdot \|_{\mathcal S_{s,h_1}} \leqs \| \cdot \|_{\mathcal S_{s,h_2}}$ when $h_1 \geqs h_2 > 0$
this shows that any seminorm $\| \cdot \|_{\mathcal S_{s,h}}$ can be estimated by a seminorm from $\{  \| \cdot \|_h \}_{h > 0}$. 
\end{proof}

The next project is to prove the fundamental Theorem \ref{thm:contGS} which shows that $\cK_t$ is uniformly continuous on $\Sigma_s(\rr d)$ for $0 \leqs t \leqs T$, for any $T > 0$. 
In order to prove it we need several auxiliary results. 
First we study the derivatives of a Gaussian type function $g_\lambda (x) = e^{\lambda x^2/2}$ for $x \in \ro$ and $\lambda \in \co$. 
It is clear that 
\begin{equation}\label{eq:gaussianderivative}
\partial^k g_\lambda (x) = p_{\lambda,k} (x) \, g_\lambda (x)
\end{equation}
where $p_{\lambda,k}$ is a polynomial of order $k \in \no$. 
This polynomial is essentially a rescaled Hermite polynomial with complex argument \cite{Szego1}. 

\begin{lem}\label{lem:gaussderivata}
Suppose $g_\lambda (x) = e^{\lambda x^2/2}$ for $x \in \ro$ and $\lambda \in \co$, 
let $p_{\lambda,k}$ be the polynomial defined in \eqref{eq:gaussianderivative} for $k \in \no$, 
and let $s > \frac{1}{2}$. 
For each $\mu > 0$ there exists $0 < \delta \leqs 1$ such that $p_{\lambda,k}$ satisfy the following estimates provided $|\lambda| \leqs \delta$: 
For any $h > 0$
\begin{equation*}
|p_{\lambda,k} (x)|
\lesssim h^k k!^s e^{\mu \, h^{-\frac{1}{s}} |x|^{\frac{1}{s}}}, \quad x \in \ro, \quad k \in \no.  
\end{equation*}
\end{lem}

\begin{proof}
By a straightforward induction argument one may confirm the formula (cf. \cite[Eq.~(5.5.4)]{Szego1})
\begin{equation*}
p_{\lambda,k} (x) = k! \sum_{m=0}^{\lfloor k/2  \rfloor} \frac{x^{k-2m} \lambda^{k-m}}{m!(k-2m)! 2^m}. 
\end{equation*}
Since $k! \leqslant 2^k (k-2m)! (2m)!$ we can estimate $|p_{\lambda,k} (x)|$ as 
\begin{equation*}
|p_{\lambda,k} (x)| 
\leqslant \sum_{m=0}^{\lfloor k/2  \rfloor} \frac{(|\lambda|^{\frac{1}{2}} \, |x|)^{k-2m} (2m)!}{m! 2^{m-k}} 
\leqslant \sum_{m=0}^{\lfloor k/2  \rfloor} ( \delta^{\frac{1}{2}} |x|)^{k-2m} m! 2^{m+k}.
\end{equation*}
Combining with $m! = m!^{2s -\ep}$ where $\ep = 2s-1 > 0$,
this gives for any $h > 0$ and $b > 0$
\begin{align*}
|p_{\lambda,k} (x)| h^{-k} k!^{-s}
& \leqs \sum_{m=0}^{\lfloor k/2  \rfloor} ( \delta^{\frac{1}{2}} \, |x|)^{k-2m} m!^{2s - \ep} 2^{m+k} h^{-k} k!^{-s} \\
& = \sum_{m=0}^{\lfloor k/2  \rfloor} 
\left(\frac{ \left( \frac{b}{s} ( \delta^{\frac{1}{2}} \, |x|) ^{\frac{1}{s}} \right)^{k-2m}}{(k-2m)!}\right)^s 
\left( \frac{b}{s} \right)^{s(2m-k)}
\left( \frac{(k-2m)! m!^2}{k!}\right)^s \frac{2^{m+k} h^{-k}}{m!^\ep} \\
& \leqs e^{b \, \delta^{\frac{1}{2s}} |x|^{\frac{1}{s}} }
\left( 2 \left( \frac{s}{b} \right)^s h^{-1} \right)^{k}
\sum_{m=0}^{\lfloor k/2  \rfloor}
\left(\frac{ \left( 2 \left( \frac{b}{s} \right)^{2s} \right)^{\frac{m}{\ep}}}{m!}\right)^\ep \\
& \leqs 
e^{\ep \left(2 \left( \frac{b}{s} \right)^{2s} \right)^{\frac{1}{\ep}} }
e^{b \, \delta^{\frac{1}{2s}} |x|^{\frac{1}{s}} }
\left( 4 \left( \frac{s}{b} \right)^s h^{-1} \right)^{k} \\
& = 
C_{s,b} \, e^{b \, \delta^{\frac{1}{2s}} |x|^{\frac{1}{s}} }, 
\end{align*}
where $C_{s,b} > 0$, provided $b = s \, 4^{\frac{1}{s}} h^{-\frac{1}{s}}$. 
Thus if $\delta \leqs 4^{-2} \left( \frac{\mu}{s} \right)^{2s}$ then
\begin{equation*}
b \, \delta^{\frac{1}{2s}} = s \, 4^{\frac{1}{s}} \delta^{\frac{1}{2s}} h^{-\frac{1}{s}}
\leqs \mu h^{-\frac{1}{s}}
\end{equation*}
and therefore
\begin{equation*}
|p_{\lambda,k} (x)| 
\lesssim h^{k} k!^{s} e^{\mu \, h^{-\frac{1}{s}} |x|^{\frac{1}{s}} }. 
\end{equation*}
\end{proof}

\begin{cor}\label{cor:gaussdiagonal}
Let $\lambda > 0$ and $s > \frac{1}{2}$.
Suppose $\Lambda \in \cc {2d \times 2d}$ is a diagonal matrix with entries $\lambda_j$ that are bounded as $| \lambda_j | \leqs \lambda$
for all $1 \leqs j \leqs 2d$. 
If $g(z) = e^{\frac{1}{2} \la \Lambda z, z\ra}$, $z \in \rr {2d}$, 
then 
\begin{equation}\label{eq:gaussderivative2d}
\pd \alpha g(z) = p_{\Lambda,\alpha} (z) g(z), \quad \alpha \in \nn{2d}, 
\end{equation}
where $p_{\Lambda,\alpha}$ are polynomials of order $|\alpha|$.
For each $\mu > 0$ there exists $0 < \delta \leqs 1$ such that the polynomials $p_{\Lambda,\alpha} $ satisfy the following estimates provided $\lambda \leqs \delta$:
For any $h > 0$ 
\begin{equation}\label{eq:polynomialestimate1}
|p_{\Lambda,\alpha}(z)| 
\lesssim h^{|\alpha|} \alpha!^s e^{\mu \, h^{-\frac{1}{s}} |z|^{\frac{1}{s}}}, \quad z \in \rr {2d}, \quad \alpha \in \nn {2d}.
\end{equation}
\end{cor}

\begin{prop}\label{prop:gaussianestimates}
Let $\lambda > 0$ and $\ep > 0$.
Suppose $T_t \in \cc {2d \times 2d}$, $0 \leqs t \leqs \ep$, is a parametrized family of symmetric matrices such that
for all $t \in [0,\ep]$ we have
$\re T_t \leqs 0$, and $\re T_t$ and $\im T_t$ both have eigenvalues in the interval $[-\lambda,\lambda]$.  
Let $a_t(z) = e^{\frac{1}{2} \la T_t z, z\ra}$, $z \in \rr {2d}$ and let $s > \frac{1}{2}$. 
For each $\mu > 0$ there exists $\delta > 0$ such that 
if $\lambda \leqs \delta$
then for any $h > 0$
\begin{equation}\label{eq:gaussderivative1}
|\pd \alpha a_t(z)| 
\lesssim 
h^{|\alpha|} \alpha!^s e^{\mu \, h^{- \frac{1}{s}} |z|^{\frac{1}{s}}}, \quad z \in \rr {2d}, \quad \alpha \in \nn {2d}, \quad 0 \leqs t \leqs \ep. 
\end{equation}
\end{prop}

\begin{proof}
We may factorize $\re T_t = U_t^T \Lambda_t U_t$ where $U_t \in \On(2d)$ and $\Lambda_t \in  \rr {2d \times 2d}$
is diagonal, with the non-positive eigenvalues of $\re T_t$ on the diagonal.  
The coefficients of $U_t$ satisfy the bound 
\begin{equation}\label{eq:coefficientbound}
|(U_t)_{j,k}| \leqs \| U_t \| = 1, \quad 1 \leqs j,k \leqs 2d,  
\end{equation}
where $\| U_t \|$ denotes the operator matrix norm. 

Thus $a_{t,1}(z) = e^{\frac{1}{2} \la \re T_t z, z\ra}  = g_{t,1} (U_t z)$ where $g_{t,1}$ satisfies the assumptions of Corollary \ref{cor:gaussdiagonal}. 
We pick $\delta > 0$ so that the polynomials $p_{\Lambda,\alpha,t,1}$, that correspond to $g_{t,1}$ as in \eqref{eq:gaussderivative2d}, satisfy 
\eqref{eq:polynomialestimate1} with $\mu$ replaced by $\mu_1 = \mu \, 2^{-2-\frac{2}{s}} d^{-1-\frac{1}{s}}$. 
We have
\begin{equation*}
\partial_j a_{t,1}(z) = \sum_{k=1}^{2d} (U_t)_{k,j} \partial_k g_{t,1}(U_tx), \quad 1 \leqs j \leqs 2d.  
\end{equation*}
Taking into account \eqref{eq:coefficientbound}, 
it follows that we may express $\pd \alpha a_{t,1}(z)$ for $\alpha \in \nn {2d}$ as a sum of $(2d)^{|\alpha|}$ terms, 
consisting of coefficients the modulus of which are upper bounded by one, times $\pd \beta g_{t,1} (U_t x)$
where $\beta \in \nn{2d}$ satisfies $|\beta| = |\alpha|$. 

Let $h > 0$. 
We obtain using Corollary \ref{cor:gaussdiagonal}, \cite[Eq.~(0.3.3)]{Nicola1} and the assumption $\re T_t \leqs 0$
\begin{align*}
|\pd \alpha a_{t,1}(z)| 
& \leqs  (2d)^{|\alpha|} \max_{|\beta| = |\alpha|} |\pd \beta g_{t,1} (U_t z)| \\
& \lesssim
(2 d h)^{|\alpha|} |\alpha|!^s e^{\mu_1 \, h^{- \frac{1}{s}} |U_t z|^{\frac{1}{s}}} | g_{t,1}(U_t z) |  \\
& \leqs ( (2 d)^{1+s} h)^{|\alpha|} \alpha!^s e^{\mu_1 \, h^{- \frac{1}{s}} |z|^{\frac{1}{s}}} e^{\frac{1}{2} \la \re T_t z, z\ra} \\
& \leqs ( (2 d)^{1+s} h)^{|\alpha|} \alpha!^s e^{\mu_1 \, h^{- \frac{1}{s}} |z|^{\frac{1}{s}}}, \quad 0 \leqs t \leqs \ep.  
\end{align*}

We apply the same argument to $a_{t,2}(z) = e^{\frac{i}{2} \la \im T_t z, z\ra}$. 
This gives new matrices $U_t \in \On(2d)$ and 
$a_{t,2}(z) = g_{t,2} (U_t z)$ where $g_{t,2}$ again satisfies the assumptions of Corollary \ref{cor:gaussdiagonal}. 
We obtain
\begin{align*}
|\pd \alpha a_{t,2}(z)| 
& \lesssim
( (2 d)^{1+s} h)^{|\alpha|} \alpha!^s e^{\mu_1 \, h^{- \frac{1}{s}} |z|^{\frac{1}{s}}} | a_{t,2}(z) |  \\
& = ( (2 d)^{1+s} h)^{|\alpha|} \alpha!^s e^{\mu_1 \, h^{- \frac{1}{s}} |z|^{\frac{1}{s}}}, \quad 0 \leqs t \leqs \ep.
\end{align*}
Finally Leibniz' rule gives 
\begin{align*}
|\pd \alpha a_t(z)| 
& = |\pd \alpha \left( a_{t,1} (z) \, a_{t,2}(z) \right)| \\
& \leqs \sum_{\beta \leqs \alpha} \binom{\alpha}{\beta} |\partial^{\alpha-\beta} a_{t,1}(z) | \, | \pd \beta a_{t,2}(z)| \\
& \lesssim \sum_{\beta \leqs \alpha} \binom{\alpha}{\beta} 
( (2 d)^{1+s} h)^{|\alpha-\beta|+|\beta|} (\alpha-\beta)!^s \beta!^s e^{2 \mu_1 \, h^{- \frac{1}{s}}  |z|^{\frac{1}{s}}} \\
& \leqs ( 2^{2+s} d^{1+s} h)^{|\alpha|} \alpha!^s e^{2 \mu_1 \, h^{- \frac{1}{s}} |z|^{\frac{1}{s}}}, \quad 0 \leqs t \leqs \ep. 
\end{align*}
The result now follows by replacing $2^{2+s} d^{1+s} h$ by $h$. 
\end{proof}

\begin{lem}\label{lem:symbolSTFTestimates}
Let $\ep > 0$ and $s > \frac{1}{2}$. 
Suppose that $a_t \in C^{\infty}(\rr {2d})$ is a family of functions parametrized by $t \in [0,\ep]$ 
that for any $h > 0$ satisfy the estimates 
\begin{equation*}
|\pd \alpha a_t(z)| 
\lesssim 
h^{|\alpha|} \alpha!^s e^{\mu \, h^{- \frac{1}{s}} |z|^{\frac{1}{s}}}, \quad z \in \rr {2d}, \quad \alpha \in \nn {2d}, \quad 0 \leqs t \leqs \ep, 
\end{equation*}
where $\mu = s \, 2^{ -4-\frac{3}{2s} } d^{- \frac{1}{2s} }$. 
Let $\Phi \in \Sigma_s(\rr {2d}) \setminus 0$. 
Then for any $b > 0$ there exists $C_b > 0$ such that
\begin{equation*}
\left| V_\Phi a_t (z,\zeta) \right|
\leqs C_b \, e^{\frac{b}{4} |z|^{\frac{1}{s}} - b |\zeta|^{\frac{1}{s}}}, \quad z, \zeta \in \rr {2d}, \quad 0 \leqs t \leqs \ep.   
\end{equation*}
\end{lem}

\begin{proof}
We will use the fact that 
\begin{equation*}
f \mapsto \sup_{x \in \rr d, \ \beta \in \nn d} \beta!^{-s} A^{|\beta|} e^{A |x|^{\frac{1}{s}}} |\pd \beta f(x)|
\end{equation*}
for all $A > 0$ is a family of seminorms for $\Sigma_s(\rr d)$, equivalent to \eqref{eq:seminormSigmas} for all $h > 0$
(cf. \cite[Proposition~3.1]{Carypis1}). 

Integration by parts and \eqref{eq:exppeetre1}
gives for any $h_1, h_2 > 0$
\begin{align*}
\left| \zeta^\alpha V_\Phi a_t (z,\zeta) \right|
& = (2 \pi)^{-d}  \left| \int_{\rr {2d}} a_t(w) \pdd w \alpha \left( e^{- i \la \zeta, w\ra} \right) \overline{\Phi(w-z)} \, \dd w  \right| \\
& \leqs (2 \pi)^{-d} \sum_{\beta \leqs \alpha} \binom{\alpha}{\beta}  
\int_{\rr {2d}} \left| \pd \beta a_t (w) \right| \, \left| \partial^{\alpha - \beta} \Phi (w-z) \right| \, \dd w \\
& \lesssim \sum_{\beta \leqs \alpha} \binom{\alpha}{\beta}
h_1^{|\beta|} h_2^{|\alpha-\beta|} \beta!^s (\alpha-\beta)!^s 
\int_{\rr {2d}} e^{\mu \, h_1^{- \frac{1}{s}}  |w|^{\frac{1}{s}} } e^{ - h_2^{-1} |w-z|^{\frac{1}{s}}} \, \dd w \\
& \leqs
\alpha!^s e^{2 \mu \, h_1^{- \frac{1}{s}} |z|^{\frac{1}{s}} }
\sum_{\beta \leqs \alpha} \binom{\alpha}{\beta}
h_1^{|\beta|} h_2^{|\alpha-\beta|}  
\int_{\rr {2d}} e^{ (2 \mu \, h_1^{- \frac{1}{s}} - h_2^{-1}) |w-z|^{\frac{1}{s}}} \, \dd w \\
& \lesssim 
\alpha!^s (h_1+h_2)^{|\alpha|} e^{2 \mu \, h_1^{- \frac{1}{s}} |z|^{\frac{1}{s}} }, \quad z,\zeta \in \rr {2d}, \quad \alpha \in \nn {2d}, \quad 0 \leqs t \leqs \ep, 
\end{align*}
provided $h_2^{-1} > 2  \mu \, h_1^{- \frac{1}{s}}$. 

Let $b > 0$. 
Using 
\begin{equation*}
|\zeta|^n \leqs ( 2 d) ^{\frac{n}{2}} \max_{|\alpha| = n} |\zeta^\alpha|
\end{equation*}
we obtain 
\begin{align*}
e^{\frac{b}{s} |\zeta|^{\frac{1}{s}}} \left| V_\Phi a_t (z,\zeta) \right|^{\frac{1}{s}}
& = 
\sum_{n=0}^{\infty} 2^{-n} n!^{-1} \left( \frac{2 b}{s} |\zeta|^{\frac{1}{s}} \right)^n 
\left| V_\Phi a_t (z,\zeta) \right|^{\frac{1}{s}} \\
& \leqs 
2 \left( \sup_{n \geqs 0}  n!^{-s} \left( \left( \frac{2 b}{ s} \right)^s |\zeta| \right)^n  \left| V_\Phi a_t (z,\zeta) \right| \right)^{\frac{1}{s}} \\
& \lesssim 
\left( \sup_{n \geqs 0}  \left( \left( \frac{2 b}{s} \right)^s (2 d)^{\frac{1}{2}} \right)^n \max_{|\alpha| = n} 
\frac{\left| \zeta^\alpha V_\Phi a_t (z,\zeta) \right|}{n!^s} \right)^{\frac{1}{s}} \\
& \lesssim 
e^{\frac{1}{s} 2 \mu \, h_1^{- \frac{1}{s}} |z|^{\frac{1}{s}} }
\left( \sup_{n \geqs 0}  \left( \left( \frac{2 b}{s} \right)^s (2 d)^{\frac{1}{2}} (h_1+h_2) \right)^n  \right)^{\frac{1}{s}}. 
\end{align*}
The result now follows provided the following three conditions are true: 
\begin{align}
& 2 \mu \, h_1^{- \frac{1}{s}} 
= \frac{b}{4}, \label{eq:cond1} \\
& h_2^{-1} > 2 \mu \, h_1^{- \frac{1}{s}} , \label{eq:cond2} \\ 
& \left( \frac{2 b}{s} \right)^s (2 d)^{\frac{1}{2}} (h_1+h_2) \leqs 1. \label{eq:cond3}
\end{align}
We first pick 
\begin{equation*}
h_1 = \frac{1}{2} \left( \frac{s}{2b} \right)^s (2 d)^{-\frac{1}{2}}
= s^s 2^{-\frac{3}{2}  - s} d^{-\frac{1}{2}} b^{-s}
\end{equation*}
which means that \eqref{eq:cond1} is satisfied. 
Since
\begin{equation*}
\left( \frac{2 b}{s} \right)^s (2 d)^{\frac{1}{2}} h_1  = \frac{1}{2}, 
\end{equation*}
we may pick $h_2 > 0$ sufficiently small so that \eqref{eq:cond2} and \eqref{eq:cond3}
are satisfied. 
\end{proof}

Finally we are in a position to prove that estimates for a family of symbols as required in Lemma \ref{lem:symbolSTFTestimates}
give rise to operators that are uniformly bounded on $\Sigma_s(\rr d)$. 
It is interesting to compare this result with 
\cite[Theorem~4.10]{Cappiello2}.
The conditions that are sufficient for continuity given in \cite[Theorem~4.10]{Cappiello2} and here are quite similar,  
but neither condition implies the other. 

\begin{prop}\label{prop:continuitySigmas}
Suppose $s > \frac{1}{2}$ and $\ep > 0$. 
Let $a_t \in C^{\infty}(\rr {2d})$ be a family of functions parametrized by $t \in [0,\ep]$, 
that for any $h > 0$ satisfy the estimates 
\begin{equation*}
|\pd \alpha a_t(z)| 
\lesssim 
h^{|\alpha|} \alpha!^s e^{\mu \, h^{-\frac{1}{s}} |z|^{\frac{1}{s}}}, \quad z \in \rr {2d}, \quad \alpha \in \nn {2d}, \quad 0 \leqs t \leqs \ep, 
\end{equation*}
where $\mu = s \, 2^{ -4-\frac{3}{2s} } d^{- \frac{1}{2s} }$. 
Then for any $h > 0$ there exists $h_1 = h_1(h) > 0$ and $C = C_h > 0$ such that 
\begin{equation*}
\| a_t^w(x,D) f \|_h \leqs C \| f \|_{h_1}, \quad 0 \leqs t \leqs \ep, \quad f \in \Sigma_s(\rr d). 
\end{equation*}
\end{prop}

\begin{proof}
Let $\fy \in \Sigma_s(\rr d)$ be such that $\Phi = W(\fy,\fy) \in \Sigma_s(\rr {2d})$ satisfies $\| \Phi \|_{L^2} = 1$. 
We use the Weyl quantization formula \eqref{eq:wignerweyl}, involving the 
Wigner distribution \eqref{eq:wignerdistribution}, and \eqref{eq:STFTinverse}.  
This gives for $f,g \in \Sigma_s(\rr d)$ and $w \in \rr {2d}$
\begin{equation}\label{eq:weylSTFT}
(a_t^w(x,D) f, \Pi(w) g) 
= (2 \pi)^{-\frac{d}{2}}  (a_t, W( \Pi(w) g, f) )
= (2 \pi)^{-\frac{d}{2}}  (V_{\Phi} a_t, V_\Phi W( \Pi(w) g, f) ). 
\end{equation}

Since $\Phi = W(\fy,\fy)$ we obtain from \cite[Lemma~14.5.1 and Lemma~3.1.3]{Grochenig1}
\begin{align*}
|V_\Phi W( \Pi(w) \fy, f) (z,\zeta) | 
& = \left| V_\fy f \left(z + \frac1{2} \J \zeta \right) \right| \, \left| V_\fy \left(\Pi(w) \fy  \right) \left(z - \frac1{2} \J \zeta \right) \right| \\ 
& = \left| V_\fy f \left(z + \frac1{2} \J \zeta \right) \right| \, \left| V_\fy \fy  \left(z - w - \frac{1}{2} \J \zeta \right) \right|. 
\end{align*}
Inserting this into \eqref{eq:weylSTFT}, using 
Lemma \ref{lem:symbolSTFTestimates}
and \eqref{eq:exppeetre1} we obtain for any $b > 0$
\begin{equation}\label{eq:STFTestimate1}
\begin{aligned}
|V_\fy (a_t^w(x,D) f) (w) |
& = (2 \pi)^{-\frac{d}{2}}  |(a_t^w(x,D) f, \Pi(w) \fy)| \\
& \leqs  (2 \pi)^{-d} \int_{\rr {4d}}  |V_{\Phi} a_t (z,\zeta) | \,  | V_\Phi W( \Pi(w) \fy, f) (z,\zeta) | \, \dd z \, \dd \, \zeta \\
& \leqs C_b \int_{\rr {4d}} e^{ \frac{b}{4} |z|^{\frac{1}{s}} - b | \zeta |^{\frac{1}{s}} } \,  \left| V_\fy f \left(z + \frac1{2} \J \zeta \right) \right| \, \left| V_\fy \fy  \left(z - w - \frac{1}{2} \J \zeta \right) \right| \, \dd z \, \dd \zeta \\
& = C_b \int_{\rr {4d}}  e^{ \frac{b}{4} \left| z -  \frac1{2} \J \zeta \right|^{\frac{1}{s}} - b | \zeta |^{\frac{1}{s}} } \,  \left| V_\fy f (z) \right| \, \left| V_\fy \fy  \left(z - w - \J \zeta \right) \right| \, \dd z \, \dd \zeta \\
& \leqs C_b \int_{\rr {4d}}  e^{ \frac{b}{2} \left| z \right|^{\frac{1}{s}} - b \left( 1 - 2^{-1 -\frac{1}{s}} \right) | \zeta |^{\frac{1}{s}} } \,  \left| V_\fy f (z) \right| \, \left| V_\fy \fy \left(z - w - \J \zeta \right) \right| \, \dd z \, \dd \zeta. 
\end{aligned}
\end{equation}
The estimate is uniform with respect to $t \in [0,\ep]$. 

Next we use the seminorms on $\Sigma_s(\rr d)$ defined by 
\begin{equation}\label{eq:seminormGS3}
\Sigma_s(\rr d) \ni f \mapsto \| f \|_A'' =  \sup_{z \in \rr {2d}} e^{A |z|^{\frac{1}{s}}} |V_\fy f (z) |, \quad A > 0, 
\end{equation}
where $\fy \in \Sigma_s(\rr d) \setminus \{ 0 \}$ is fixed but arbitrary
(cf. \cite[Propositon~3.1]{Carypis1}). 
Using $2^{-1 - \frac{1}{s}} < 2^{-1}$ and again \eqref{eq:exppeetre1}
we obtain for any $a > 0$
\begin{equation}\label{eq:STFTestimate2}
\begin{aligned}
& |V_\fy (a_t^w(x,D) f) (w) | \\
& \leqs C_b \| f \|_{b + 2 a}'' \| \fy \|_{4 a}'' 
\int_{\rr {4d}} e^{-\left( \frac{b}{2} + 2 a \right) \left| z \right|^{\frac{1}{s}} - b \left( 1 - 2^{-1 -\frac{1}{s}} \right) | \zeta |^{\frac{1}{s}}  - 4 a \left| z - w - \J \zeta \right|^{\frac{1}{s}} }  \, \dd z \, \dd \zeta \\
& \leqs C_b \| f \|_{b + 2 a}'' \| \fy \|_{4 a}'' 
\int_{\rr {4d}} e^{-\left( \frac{b}{2} + 2 a \right) \left| z \right|^{\frac{1}{s}} - \frac{b}{2} | \zeta |^{\frac{1}{s}}  - 4 a \left| z - w - \J \zeta \right|^{\frac{1}{s}} }  \, \dd z \, \dd \zeta \\
& \leqs C_b \| f \|_{b + 2a}'' \| \fy \|_{4 a}'' 
\int_{\rr {4d}} e^{-\left( \frac{b}{2} + 2 a \right) \left| z \right|^{\frac{1}{s}} - \left( \frac{b}{2}  - 4 a \right) | \zeta |^{\frac{1}{s}}  - 2a  \left| z - w \right|^{\frac{1}{s}} }  \, \dd z \, \dd \zeta \\
& \leqs C_b \| f \|_{b + 2 a}'' \| \fy \|_{4 a}'' 
\, e^{- a \left| w \right|^{\frac{1}{s}} }
\int_{\rr {4d}} e^{-\frac{b}{2}  \left| z \right|^{\frac{1}{s}} - \left( \frac{b}{2}  - 4 a \right) | \zeta |^{\frac{1}{s}}  }  \, \dd z \, \dd \zeta, \quad w \in \rr {2d}. 
\end{aligned}
\end{equation}
Let $B > 0$ be arbitrary. If we first pick $a \geqs B$ and then $b > 8 a$ we obtain 
\begin{equation}\label{eq:STFTestimate3}
\begin{aligned}
\| a_t^w(x,D) f \|_{B}'' 
& = \sup_{w \in \rr {2d}} e^{B | w |^{\frac{1}{s}}} |V_\fy (a_t^w(x,D) f) (w) |
& \leqs C \| f \|_{b + 2 a}'' 
\end{aligned}
\end{equation}
for a constant $C > 0$ and for all $t \in [0,\ep]$. 

Finally we combine Lemma \ref{lem:seminorms} and \cite[Proposition~3.1]{Carypis1}, which admits the conclusion that 
the seminorms \eqref{eq:seminormGS3} are equivalent to the seminorms $\| \cdot \|_h$ for $h > 0$, defined in \eqref{eq:seminormL2}. 
This implies the claim. 
\end{proof}

We have reached a point at which we may prove the theorem for which Lemma \ref{lem:gaussderivata}, 
Corollary \ref{cor:gaussdiagonal}, Proposition \ref{prop:gaussianestimates}, 
Lemma \ref{lem:symbolSTFTestimates}, and Proposition \ref{prop:continuitySigmas} are preparations. 

\begin{thm}\label{thm:contGS}
Let $\re Q \geqs 0$, $s > \frac{1}{2}$ and $T>0$. 
For every $h > 0$ there exists $h_1 = h_1(h) > 0$ and $C = C_{T,h} > 0$ such that 
\begin{equation*}
\| \cK_t f \|_{h} \leqs C \| f \|_{h_1}, \quad 0 \leqs t \leqs T, \quad f \in \Sigma_s(\rr d). 
\end{equation*}
\end{thm}

\begin{proof}
It suffices to show the following statement. 
There exists $\ep > 0$ such that for any 
$h > 0$ there exists $h_1 = h_1 (h) > 0$ and $C = C (h) > 0$ such that
\begin{equation}\label{eq:contsmallt}
\| \cK_t f \|_{h} \leqs C \| f \|_{h_1}, \quad 0 \leqs t \leqs \ep. 
\end{equation}
In fact, suppose that \eqref{eq:contsmallt} holds, 
for given $\ep > 0$, all $h > 0$ and some $C, h_1 > 0$. 
Take $n \in \no$ such that $ n \geqs T \ep^{-1}$, which implies $t/n \leqs \ep$ for $0 \leqs t \leqs T$. 
We use the semigroup property $\cK_{t_1+t_2} = \cK_{t_1} \cK_{t_2}$ for $t_1, t_2 \geqs 0$. 
Thus we obtain 
from \eqref{eq:contsmallt} the existence of $C_1, C_2, \cdots C_n > 0$ and $h_1, h_2, \cdots h_n > 0$
\begin{align*}
\| \cK_t f \|_{h} 
& = \| (\cK_{t/n})^n f \|_{h}
\leqs C_1 \| (\cK_{t/n})^{n-1} f \|_{h_1}
\leqs C_1 C_2 \| (\cK_{t/n})^{n-2} f \|_{h_2} \\
& \leqs C_1 C_2 \cdots C_n \| f \|_{h_n}, \quad 0 \leqs t \leqs T, 
\end{align*}
which implies the claim of the theorem. 

Thus we may concentrate on the proof of \eqref{eq:contsmallt} for some $\ep > 0$, and for all $h > 0$, some $h_1 = h_1 (h) > 0$ and some $C = C(h) > 0$. 
We express $\cK_t$ as a Weyl operator \eqref{eq:weylquantization} as $\cK_t = a_t^w(x,D)$. 
Then we can benefit from H\"ormander's \cite[Theorem~4.3]{Hormander2} explicit formula for the Weyl symbol 
\begin{equation*}
a_t(z) = ( \det (\cos (t F)))^{-\frac{1}{2}} \, \exp \left( \sigma(\tan(t F) z,z) \right), \quad z \in \rr {2d}, 
\end{equation*}
where $F = \J Q$ and $\tan(tF) = \sin(t F) (\cos(tF))^{-1}$, 
which is valid for all $t \geqs 0$ such that $\det (\cos (t F)) \neq 0$. 
According to \cite[Theorem~4.1]{Hormander2}, $\det (\cos (t F)) \neq 0$ 
unless $t \lambda \in \pi \left( \frac{1}{2} + \zo \right)$ where $\lambda \in \co$ is an eigenvalue of $F$. 
Clearly it is possible to pick $\ep > 0$ such that $\det (\cos (t F)) \neq 0$ for $0 \leqs t \leqs \ep$. 

The exponent of $a_t$ is 
\begin{equation*}
\sigma(\tan(t F) z,z) = \la \J \tan(t F) z, z \ra
= \frac{1}{2} \la T_t z, z \ra
\end{equation*}
where the symmetric matrix $T_t \in \cc {2d \times 2d}$ is 
\begin{equation*}
T_t =  \J \tan(tF) - (\tan(tF) )^T \J 
\end{equation*}
due to $\J^T = - \J$. 

Since $\cos (t F) \to I$ as $t \to 0^+$ we may assume that
the factor $( \det (\cos (t F)))^{-\frac{1}{2}}$ satisfies 
\begin{equation*}
( \det (\cos (t F)))^{-\frac{1}{2}} \leqs 2, \quad 0 \leqs t \leqs \ep, 
\end{equation*}
after possibly decreasing $\ep > 0$. 

According to \cite[Theorem~4.6]{Hormander2} we have $\re T_t \leqs 0$ for $t \in [0,\ep]$. 
Since $T_t \to 0$ as $t \to 0^+$, we may assume that $\re T_t$ and $\im T_t$ both have small 
eigenvalues, uniformly over $t \in [0,\ep]$, again after possibly decreasing $\ep > 0$. 
Specifically we assume that the eigenvalues belong to $[-\delta,\delta]$ 
for $t \in [0,\ep]$, where $\delta > 0$ is chosen small enough to guarantee by Proposition \ref{prop:gaussianestimates} that 
the estimates \eqref{eq:gaussderivative1} hold for all $h > 0$ with $\mu = s \, 2^{ -4-\frac{3}{2s} } d^{- \frac{1}{2s} }$. 
The claim is now a consequence of Proposition \ref{prop:continuitySigmas}. 
\end{proof}

By Theorem \ref{thm:contGS} we may extend $\cK_t$ uniquely from the domain $\Sigma_s(\rr d)$ to $\Sigma_s'(\rr d)$
by the assignment 
\begin{equation}\label{eq:dualsemigroupGS}
(\cK_t u, \fy) = (u, \cK_t^* \fy) = (u, \cK_{e^{- 2 i t \overline F}} \fy), \quad u \in \Sigma_s'(\rr d), \quad \fy \in \Sigma_s(\rr d). 
\end{equation}

\begin{cor}\label{cor:propagatorcontSigma}
If $s > \frac1{2}$ and $t \geqs 0$ then $\cK_t$ is a continuous linear operator on $\Sigma_s(\rr d)$, 
that extends uniquely to a continuous linear operator on $\Sigma_s'(\rr d)$ equipped with its weak$^*$ topology. 
\end{cor}

Theorem \ref{thm:contGS} implies in particular that $\cK_t = \cK_{e^{- 2 i t F}}: \Sigma_s(\rr d) \to \Sigma_s(\rr d)$
is continuous for each fixed $t \geqs 0$. 

\begin{rem}\label{rem:generalization}
The continuity of $\cK_t: \Sigma_s(\rr d) \to \Sigma_s(\rr d)$ can be generalized as follows. 
The operator $\cK_T: \Sigma_s(\rr d) \to \Sigma_s(\rr d)$ is continuous for any matrix $T \in \Sp(d,\co)$ 
which is positive in the sense of 
\begin{equation}\label{eq:possympmatrix}
i \left( \sigma(\overline{TX}, TX) - \sigma(\overline{X},X) \right) \geqslant 0, \quad X \in T^* \cc d, 
\end{equation}
(cf. \cite{Hormander2}), 
where $\cK_T$ is the operator with kernel $K_T$ defined as in \eqref{schwartzkernel1} with $e^{- 2 i t F}$ replaced by $T$. 
Condition \eqref{eq:possympmatrix} means that the graph of $T$ is a positive Lagrangian in $T^* \cc d \times T^* \cc d$. 
The operator $\cK_t$ with kernel $K_{e^{- 2 i t F}}$ defined by the oscillatory integral kernel \eqref{schwartzkernel1} is a particular case with $T = e^{- 2 i t F}$. 
The matrix $e^{- 2 i t F}$ is a positive matrix in $\Sp(d,\co)$ according to \cite[Lemma~5.2]{PRW1}. 

This generalization of Theorem \ref{thm:contGS} has been stated in \cite[Proposition~8.1]{Carypis1}. 
The proof there is unfortunately wrong but it has been corrected \cite{Carypis2}. 
\end{rem}

The next result is a Gelfand--Shilov version of Lemma \ref{lem:boundedsetinvarianceS}. 

\begin{lem}\label{lem:boundedsetoperator}
If $B \subseteq \Sigma_s(\rr d)$ is bounded and $N > 0$ is an integer then 
\begin{equation*}
\{ x^\gamma D^\kappa f, \ f \in B, \  |\gamma+\kappa| < N \} \subseteq \Sigma_s(\rr d)
\end{equation*}
is also bounded. 
\end{lem}

\begin{proof}
Using the seminorms \eqref{eq:seminormL2}
the assumption means that 
\begin{equation}\label{eq:boundedset}
\sup_{f \in B} \| f \|_h = C_h < \infty \quad \forall h > 0. 
\end{equation}
We have for $f \in B$ and $\alpha, \beta, \gamma, \kappa \in \nn d$ 
\begin{align*}
\left\| x^\alpha D ^\beta \left( x^\gamma D^\kappa f \right) \right\|_{L^2}
& = \left\| \sum_{\sigma \leqs \min( \beta,\gamma)} \binom{\beta}{\sigma} \frac{\gamma! \, i^{-|\sigma|}}{(\gamma-\sigma)!} x^{\alpha + \gamma - \sigma} D^{\kappa+\beta-\sigma} f \right\|_{L^2} \\
& \leqs \sum_{\sigma \leqs \min( \beta,\gamma)} \binom{\beta}{\sigma} \sigma! \, 2^{|\gamma|}  \left\| x^{\alpha + \gamma - \sigma} D^{\kappa+\beta-\sigma} f \right\|_{L^2}. 
\end{align*}

As in the proof of Lemma \ref{lem:seminorms} we next use $1 = 2 s - \delta$ where $\delta > 0$. 
Let $h > 0$. 
Since $\| \cdot \|_{h_1} \leqs \| \cdot \|_{h_2}$ when $h_1 \geqs h_2 > 0$ we may assume that $h \leqs 1$. 
Provided $|\gamma+\kappa| < N$
we obtain using \eqref{eq:exponentialestimate} and \eqref{eq:boundedset}
\begin{align*}
\left\| x^\alpha D ^\beta \left( x^\gamma D^\kappa f \right) \right\|_{L^2}
& \leqs 2^{N} \sum_{\sigma \leqs \min( \beta,\gamma)} \binom{\beta}{\sigma} \sigma!^{2s-\delta}  \left\| x^{\alpha + \gamma - \sigma} D^{\kappa+\beta-\sigma} f \right\|_{L^2} \\
& \leqs 2^{N} C_h \sum_{\sigma \leqs \min( \beta,\gamma)} \binom{\beta}{\sigma} \sigma!^{2s-\delta}  h^{|\alpha + \beta + \gamma + \kappa - 2 \sigma |} ((\alpha + \gamma - \sigma)! (\kappa+\beta-\sigma)!)^s \\
& \leqs 2^{N} C_h \, h^{|\alpha + \beta|} \sum_{\sigma \leqs \min( \beta,\gamma)} \binom{\beta}{\sigma} \sigma!^{-\delta}  h^{- 2 |\sigma|}  ((\alpha + \gamma)! (\kappa+\beta)!)^s \\
& \leqs 2^{N} C_{\delta,d,h} \, C_h \, h^{|\alpha + \beta|} (\alpha! \beta!)^s \sum_{\sigma \leqs \min( \beta,\gamma)} \binom{\beta}{\sigma} 2^{s |\alpha+\beta+\gamma+\kappa|} (\gamma! \kappa!)^s \\
& \leqs C_N \, C_{\delta,d,h} \, C_h \, (2^{s+1} h)^{|\alpha + \beta|} (\alpha! \beta!)^s.  
\end{align*}
This gives for some $C_{\delta,d,h,N}' > 0$
\begin{align*}
\left\| x^\gamma D^\kappa f  \right\|_{2^{s+1} h}
\leqs C_{\delta,d,h,N}', \quad |\gamma+\kappa| < N, \quad f \in B. 
\end{align*}
Since $0 < h \leqs 1$ is arbitrary we have proved the claim.
\end{proof}

The proof of the next result is omitted since it is conceptually identical to the proof of Lemma \ref{lem:boundedsetcoverS}. 

\begin{lem}\label{lem:boundedsetcover}
If $B \subseteq \Sigma_s(\rr d)$ is bounded and $\ep > 0$
then there exists $K \in \no$ and $\fy_j \in \Sigma_s(\rr d)$ for $1 \leqs j \leqs K$ such that 
\begin{equation*}
B \subseteq \bigcup_{j=1}^K B_\ep (\fy_j)
 \end{equation*}
where the open balls $B_\ep (\fy_j) \subseteq L^2(\rr d)$ refer to the $L^2$ norm. 
\end{lem}

We have now reached a point where we may prove that $\cK_t$ is a strongly continuous semigroup on $\Sigma_s(\rr d)$. 
It is a consequence of the following result. 

\begin{thm}\label{thm:strongcontGS}
The map $[0,\infty) \ni t \mapsto \cK_t$ is a semigroup on $\Sigma_s(\rr d)$, which satisfies 
for each bounded set $B \subseteq \Sigma_s(\rr d)$ and all $h > 0$
\begin{equation}\label{eq:verystrongcont}
\lim_{t \to 0^+} \sup_{\fy \in B} \| (\cK_t - I) \fy \|_h = 0. 
\end{equation}
\end{thm}

\begin{proof}
The semigroup property $\cK_{t_1 + t_2} = \cK_{t_1} \cK_{t_2}$ for $t_1, t_2 \geqs 0$, as well as $\cK_0 = I$, 
are immediate since they hold on $L^2$ and $\Sigma_s(\rr d) \subseteq L^2$, 
and Corollary \ref{cor:propagatorcontSigma} shows that $\cK_t : \Sigma_s (\rr d) \to \Sigma_s(\rr d)$ 
is continuous for each $t \geqs 0$. 

It remains to show \eqref{eq:verystrongcont} where $h > 0$ and $B \subseteq \Sigma_s(\rr d)$ is bounded
as in \eqref{eq:boundedset}. 
We may assume that $h \leqs 1$. 

Let $\ep > 0$ and $N \in \no$. 
If $|\alpha + \beta| \geqs N$ and $0 < t \leqs 1$ then we obtain from Theorem \ref{thm:contGS}
\begin{equation}\label{eq:largeindexestimate}
\begin{aligned}
\frac{\| x^\alpha D^\beta (\cK_t -I) \fy \|_{L^2}}{h^{|\alpha + \beta|} (\alpha! \beta!)^s}
& \leqs 2^{-|\alpha + \beta|} \frac{\| x^\alpha D^\beta \cK_t \fy \|_{L^2} + \| x^\alpha D^\beta \fy \|_{L^2}}{\left( \frac{h}{2} \right)^{|\alpha + \beta|} (\alpha! \beta!)^s} \\
& \leqs 2^{-N} \left( \| \cK_t \fy \|_{\frac{h}{2}} + \| \fy \|_{\frac{h}{2}} \right) \\
& \lesssim 2^{-N} \left( \| \fy \|_{h_1} + \| \fy \|_{\frac{h}{2}} \right) 
\leqs \ep, \quad \fy \in B, 
\end{aligned}
\end{equation}
for some $h_1 > 0$, 
provided $N \in \no$ is sufficiently large, taking into account \eqref{eq:boundedset}.  

We also have to consider $\alpha, \beta \in \nn d$ such that $|\alpha + \beta| < N$.  
From Lemma \ref{lem:operatorcommutator} and the contraction property of $\cK_t$ acting on $L^2$ we obtain for $0 < t \leqs 1$
\begin{equation}\label{eq:smalltimeestimate}
\begin{aligned}
\| x^\alpha D^\beta (\cK_t -I) \fy \|_{L^2}
& \leqs |C_{\alpha,\beta} (t)| \, \| (\cK_t -I) x^\alpha D^\beta \fy \|_{L^2} + |C_{\alpha,\beta} (t) - 1| \, \| x^\alpha D^\beta \fy \|_{L^2} \\
& \qquad + \sum_{\stackrel{|\gamma + \kappa| \leqs |\alpha+\beta|}{(\gamma,\kappa) \neq (\alpha,\beta)}} |C_{\gamma,\kappa} (t)| \, \| x^\gamma D^\kappa \fy \|_{L^2} \\
& \leqs C \| (\cK_t -I) x^\alpha D^\beta \fy \|_{L^2} + |C_{\alpha,\beta} (t) - 1| \, \| x^\alpha D^\beta \fy \|_{L^2} \\
& \qquad + \sum_{\stackrel{|\gamma + \kappa| \leqs |\alpha+\beta|}{(\gamma,\kappa) \neq (\alpha,\beta)}} |C_{\gamma,\kappa} (t)| \, \| x^\gamma D^\kappa \fy \|_{L^2}
\end{aligned}
\end{equation}
where $C>0$
and \eqref{eq:tlimits} hold.  

By Lemma \ref{lem:boundedsetoperator}, $\{ x^\alpha D^\beta \fy: \ \fy \in B, \  |\alpha + \beta| < N \} \subseteq \Sigma_s(\rr d)$ is bounded.  
Thus by Lemma \ref{lem:boundedsetcover} there exists $K \in \no$ and $\fy_j \in \Sigma_s(\rr d)$, $1 \leqs j \leqs K$, such that 
\begin{equation*}
\min_{1 \leqs j \leqs K} \| x^\alpha D^\beta \fy - \fy_j \|_{L^2} < \frac{\ep h^N}{8 C}, \quad  |\alpha + \beta| < N, \quad \fy \in B. 
\end{equation*}

The strong continuity and the contraction property of $\cK_t$ acting on $L^2$
gives for $0 < t \leqs \delta$
\begin{equation}\label{eq:firstterm}
\begin{aligned}
\| (\cK_t -I) x^\alpha D^\beta \fy \|_{L^2} 
& = \min_{1 \leqs j \leqs K} \| (\cK_t -I) (x^\alpha D^\beta \fy - \fy_j + \fy_j) \|_{L^2} \\
& \leqs \min_{1 \leqs j \leqs K} \left( 2 \| x^\alpha D^\beta \fy - \fy_j\|_{L^2} + \| (\cK_t -I) \fy_j \|_{L^2} \right) \\
& \leqs \frac{\ep h^N}{4 C} + \frac{\ep h^N}{4 C} =  \frac{\ep h^N}{2 C}, \quad |\alpha + \beta| < N, \quad \fy \in B, 
\end{aligned}
\end{equation}
provided $\delta > 0$ is sufficiently small. 

We have 
\begin{equation}\label{eq:secondterm}
\| x^\alpha D^\beta \fy \|_{L^2}
\leqs C_h (\alpha! \beta!)^s h^{|\alpha + \beta|}, \quad \alpha, \beta \in \nn d, \quad \fy \in B, 
\end{equation}
and for $|\gamma + \kappa| \leqs |\alpha+\beta| < N$ we have 
\begin{equation*}
\| x^\gamma D^\kappa \fy \|_{L^2} \leqs C_h h^{|\gamma+\kappa|} (\gamma! \kappa!)^s
\leqs C_h \max_{|\gamma + \kappa| < N}(\gamma! \kappa!)^s := C_{h,N}, \quad \fy \in B. 
\end{equation*}
Due to \eqref{eq:tlimits} the latter gives for $0 < t \leqs \delta$
\begin{equation}\label{eq:thirdterm}
\sum_{\stackrel{|\gamma + \kappa| \leqs |\alpha+\beta|}{(\gamma,\kappa) \neq (\alpha,\beta)}} |C_{\gamma,\kappa} (t)| \, \| x^\gamma D^\kappa \fy \|_{L^2} 
< \frac{\ep h^N}{4}, \quad |\alpha+\beta| < N, \quad \fy \in B, 
\end{equation}
after possibly decreasing $\delta > 0$. 

Finally we insert \eqref{eq:firstterm}, \eqref{eq:secondterm} and \eqref{eq:thirdterm} into \eqref{eq:smalltimeestimate}. 
Using \eqref{eq:tlimits} we obtain then for $0 < t \leqs \delta$, again after possibly decreasing $\delta > 0$, 
\begin{align*}
& \frac{\| x^\alpha D^\beta (\cK_t -I) \fy \|_{L^2}}{h^{|\alpha+\beta|} (\alpha! \beta !)^s} \\
& \leqs \frac{C \| (\cK_t -I) x^\alpha D^\beta \fy \|_{L^2}}{h^N} + |C_{\alpha,\beta} (t) - 1| \, C_h 
+ \sum_{\stackrel{|\gamma + \kappa| \leqs |\alpha+\beta|}{(\gamma,\kappa) \neq (\alpha,\beta)}} |C_{\gamma,\kappa} (t)| \, 
\frac{\| x^\gamma D^\kappa \fy \|_{L^2} }{h^N} \\
& \leqs \frac{\ep}{2} + \frac{\ep}{4} + \frac{\ep}{4} = \ep, \quad |\alpha + \beta| < N, \quad \fy \in B. 
\end{align*}
If we combine this estimate with \eqref{eq:largeindexestimate} we obtain for $0 < t \leqs \delta$
\begin{equation*}
\frac{\| x^\alpha D^\beta (\cK_t -I) \fy \|_{L^2}}{h^{|\alpha+\beta|} (\alpha! \beta !)^s} 
\leqs \ep, \quad \alpha, \beta \in \nn d, \quad \fy \in B. 
\end{equation*}
Since $\ep > 0$ is arbitrary this proves \eqref{eq:verystrongcont}. 
\end{proof}

As a consequence, picking the bounded set $B$ as a single element in $\Sigma_s$, we obtain the following result. 
The local equicontinuity is a consequence of Theorem \ref{thm:contGS}. 

\begin{cor}\label{cor:locequicont}
For $t \geqs 0$, $\cK_t$ is a locally equicontinuous strongly continuous semigroup on $\Sigma_s(\rr d)$. 
\end{cor}

We denote the generator of the semigroup $\cK_t$ acting on $\Sigma_s(\rr d)$ by $L_s$, to distinguish from the generator $A_s$ defined in  \eqref{eq:generatorMs}. 
Because of $\Sigma_s \subseteq \cS$ we have $L_s \subseteq A$, cf. \eqref{eq:generatorS}. 

By \cite[Proposition~1.3]{Komura1} the domain $D(L_s) \subseteq \Sigma_s(\rr d)$ is dense, 
and by \cite[Proposition~1.4]{Komura1} $L_s$ is a closed operator in $\Sigma_s(\rr d)$.

\begin{prop}\label{prop:generatorbounded}
The generator $L_s$ is a continuous operator on $\Sigma_s(\rr d)$ and thus $D(L_s) = \Sigma_s(\rr d)$.
\end{prop}

\begin{proof}
First we consider the Weyl symbol $q \in \Gamma^2$ defined in \eqref{eq:quadraticform}. It satisfies
\begin{align*}
|\pd \alpha q(z) | 
& \lesssim \eabs{z}^{2-|\alpha|}, \quad z \in \rr {2d}, \quad \alpha \in \nn {2d}, & |\alpha| \leqs 2, \\
& \equiv 0, & |\alpha| > 2. 
\end{align*}
Combining with (cf. \eqref{eq:exponentialestimate})
\begin{equation*}
\alpha!^{-s} h^{- |\alpha|} 
= \left( \frac{h^{- \frac{|\alpha|}{s}}}{\alpha!} \right)^s
\leqs \left( \frac{ \left( 2d h^{- \frac{1}{s}} \right)^{|\alpha|} }{|\alpha|!} \right)^s
\leqs \exp \left( 2 s d h^{-\frac{1}{s}} \right), \quad \alpha \in \nn {2d}, \quad h > 0, 
\end{equation*}
this gives
\begin{align*}
\alpha!^{-s} h^{- |\alpha|}  e^{- |z|^{\frac{1}{s}}}
|\pd \alpha q(z) | 
\lesssim \exp \left( 2 s d h^{-\frac{1}{s}} \right) \eabs{z}^{2} e^{- |z|^{\frac{1}{s}}}
\leqs C, \quad z \in \rr {2d}, \quad \alpha \in \nn {2d}, \quad h > 0, 
\end{align*}
where $C = C_{s,d,h} > 0$. 

We have proved the estimates 
\begin{equation*}
|\pd \alpha q(z) | 
\lesssim h^{|\alpha|} \alpha!^s e^{|z|^{\frac{1}{s}}}, \quad z \in \rr {2d}, \quad \alpha \in \nn {2d}, \quad \forall h > 0, 
\end{equation*}
which by \cite[Definition~2.4]{Cappiello2} implies that $q$ belongs to a space there denoted $\Gamma_{0,s}^\infty(\rr {2d})$. 
According to \cite[Theorem~4.10]{Cappiello2} the operator $q^w(x,D): \Sigma_s(\rr d) \to \Sigma_s(\rr d)$ is thereby continuous. 
Hence, referring to the seminorms \eqref{eq:seminormL2},  for any $h_1 > 0$ there exists $h_2 > 0$ such that 
\begin{equation}\label{eq:qwcontGS}
\| q^w(x,D) \fy \|_{h_1} \lesssim \| \fy \|_{h_2}, \quad \fy \in \Sigma_s(\rr d). 
\end{equation}

We have $D(L_s) \subseteq \Sigma_s(\rr d) \subseteq D( q^w(x,D) )$. 
If $f \in D(L_s)$ then the limit
\begin{equation*}
L_s f = \lim_{h \to 0^+} h^{-1} \left( \cK_h - I \right)f
\end{equation*}
exists in $\Sigma_s$. 
Since $\| \cdot \|_{L^2} \leqs \| \cdot \|_h$ for any $h > 0$, the limit also exists in $L^2$. 
It follows that $L_s f = - q^w(x,D) f$ for $f \in D(L_s)$, that is $L_s \subseteq - q^w(x,D)$. 

Let $f \in \Sigma_s(\rr d)$. By the density $D(L_s) \subseteq \Sigma_s(\rr d)$ there exists a sequence $(f_n)_{n \geqs 1} \subseteq D(L_s)$
such that $f_n \to f$ in $\Sigma_s$. 
The estimate \eqref{eq:qwcontGS} gives for any $h_1 > 0$
\begin{equation*}
\| L_s (f_n - f_m ) \|_{h_1} 
= \|  q^w(x,D) (f_n - f_m ) \|_{h_1} 
\lesssim \|  f_n - f_m  \|_{h_2}
\end{equation*}
for some $h_2 > 0$. 
Thus $( L_s f_n )_{n \geqs 1}$ is a Cauchy sequence in $\Sigma_s(\rr d)$ which converges to an element $g \in \Sigma_s(\rr d)$. 
From the closedness of $L_s$ it follows that $f \in D(L_s)$ and $L_s f = g$. 
Hence $D(L_s) = \Sigma_s(\rr d)$ and $L_s$ is continuous on $\Sigma_s(\rr d)$. 
\end{proof}

As in \eqref{eq:Aextension} we may extend $L_s$ uniquely to a continuous operator on $\Sigma_s'(\rr d)$ equipped with its weak$^*$ topology, denoted $\Sigma_{s, \rm w}'(\rr d)$. 
In fact we set, using the formal $L^2$ adjoint $L_s^* = - \overline{q}^w(x,D)$ acting on $\Sigma_s$,
\begin{equation}\label{eq:Lsextension}
(L_s u,\fy) = (u, L_s^* \fy), \quad u \in \Sigma_s'(\rr d), \quad \fy \in \Sigma_s(\rr d).  
\end{equation}

The space $\Sigma_s'(\rr d)$ equipped with its strong topology is denoted $\Sigma_{s,\rm str}'(\rr d)$, 
and the topology is defined by the seminorms 
\begin{equation*}
\Sigma_s'(\rr d) \ni u \mapsto \sup_{\fy \in B} |(u,\fy)|
\end{equation*}
for all bounded subsets $B \subseteq \Sigma_s(\rr d)$. 
Then $L_s$ defined by \eqref{eq:Lsextension} is continuous on $\Sigma_{s,\rm str}'(\rr d)$.

We can now formulate and prove a Gelfand--Shilov distribution version of Theorem \ref{thm:strongcontlocequicontS}.

\begin{thm}\label{thm:strongcontGSdist}
The semigroup $\cK_t$ is: 
\begin{enumerate}[\rm (i)] 
\item strongly continuous on $\Sigma_{s,\rm w}'$, and
\item locally equicontinuous strongly continuous on $\Sigma_{s,\rm str}'$.  
\end{enumerate}
\end{thm}

\begin{proof}
The semigroup property $\cK_{t_1 + t_2} = \cK_{t_1} \cK_{t_2}$ for $t_1, t_2 \geqs 0$, as well as $\cK_0 = I$, on $\Sigma_s'(\rr d)$
defined by \eqref{eq:dualsemigroupGS}
follow from the corresponding properties on $\Sigma_s(\rr d)$, as in the proof of Theorem \ref{thm:semigroupMs}. 

Let $u \in \Sigma_s'(\rr d)$ and let $T > 0$.  
For $0 \leqs t \leqs T$ fixed, a seminorm of $\cK_t u$ considered as an element in $\Sigma_{s,\rm str}'$ is defined by a bounded set $B \subseteq \Sigma_s(\rr d)$ as 
\begin{equation*}
\sup_{\fy \in B} |(\cK_t u,\fy)| = \sup_{\fy \in B} |(u, \cK_t^* \fy)|
\end{equation*}
and the right-hand side is a seminorm of $u$, since $\cK_t^* B \subseteq \Sigma_s(\rr d)$ is a bounded set according to 
Theorem \ref{thm:contGS}. 
This shows the continuity $\cK_t: \Sigma_{s,\rm str}'(\rr d) \to \Sigma_{s,\rm str}'(\rr d)$ 
as well as the continuity $\cK_t: \Sigma_{s,\rm w}'(\rr d) \to \Sigma_{s,\rm w}'(\rr d)$
for each fixed $t$ such that $0 \leqs t \leqs T$.
Theorem \ref{thm:contGS} also shows that $\{ \cK_t^* B, \ 0 \leqs t \leqs T \} \subseteq \Sigma_s(\rr d)$ is a bounded set
so $\cK_t$ is locally equicontinuous on $\Sigma_{s,\rm str}'(\rr d)$. 

Finally let $B \subseteq \Sigma_s(\rr d)$ be a bounded set and let $u \in \Sigma_s'(\rr d)$. 
For some $h > 0$ we obtain using Theorem \ref{thm:strongcontGS} 
\begin{align*}
\sup_{\fy \in B} |( (\cK_t -I) u,\fy)| 
= \sup_{\fy \in B} |(u, (\cK_t^*-I) \fy)|
\lesssim \sup_{\fy \in B} \| (\cK_t^*-I) \fy \|_h
\longrightarrow 0, \quad t \longrightarrow 0^+, 
\end{align*}
which shows that $\cK_t$ is strongly continuous on $\Sigma_{s,\rm str}'(\rr d)$
as well as on $\Sigma_{s,\rm w}'(\rr d)$.
\end{proof}

The generator of the semigroup $\cK_t$ acting on $\Sigma_{s,\rm w}'(\rr d)$ is denoted $L_{\rm w}'$, 
and the generator of the semigroup $\cK_t$ acting on $\Sigma_{s,\rm str}'(\rr d)$ is denoted $L_{\rm str}'$. 
By \cite[Proposition~2.1]{Komura1} we have $L_{\rm w}' = L_s$ defined by \eqref{eq:Lsextension}, 
and hence $D(L_{\rm w}') = \Sigma_s'$. 

The argument that proves $A_{\rm str}' = A_{\rm w}'$ after Theorem \ref{thm:strongcontlocequicontS} again shows that $L_{\rm str}' = L_{\rm w}'$. 
Again we may thus conclude that
the two semigroups have identical generators. 
Denoting $L' = L_{\rm str}' = L_{\rm w}'$ we have $D(L') = \Sigma_{s}'$.
We may again invoke \cite[Proposition~1.2]{Komura1} and \cite[pp.~483--84]{Kato1} to yield the following result
which is conceptually similar to Corollary \ref{cor:uniquenessCP3}.  
Note that the uniqueness space is again larger than the solution space: 
$C^1([0,\infty),\Sigma_{s,\rm str}') \subseteq C^1([0,\infty),\Sigma_{s,\rm w}')$. 

\begin{cor}\label{cor:uniquenessCP4}
For any $u_0 \in \Sigma_s' (\rr d)$
the Cauchy problem \eqref{eq:CP} has the solution $\cK_t u_0$ in the space 
$C^1([0,\infty),\Sigma_{s,\rm str}')$. 
The solution is unique in the space $C^1([0,\infty),\Sigma_{s,\rm w}')$. 
\end{cor}

There is also an alternative way to show $D(L_{\rm str}' ) = \Sigma_s'$, cf. Remark \ref{rem:abstractdual}. 
In fact, if we can show that $\Sigma_s(\rr d)$ is a reflexive space
then \cite[Theorem 1 and its Corollary]{Komura1} show that 
$\cK_t$, considered as a strongly continuous semigroup on $\Sigma_{s,\rm w}'$, is necessarily also 
strongly continuous on $\Sigma_{s,\rm str}'$, and the two semigroups have identical generators. 

Thus it remains to show that $\Sigma_s(\rr d)$ is a reflexive space (cf. \cite[Theorem~I.6.2]{Gelfand2}), 
which may be of independent interest. 
A locally convex space $X$ is called reflexive provided $X \mapsto (X_\beta')_\beta '$ is a topological isomorphism
\cite[p.~144]{Schaefer1}. Here $X_\beta'$ denotes the dual of $X$, equipped with its strong topology. 

\begin{prop}\label{prop:GSreflexive}
If $s > \frac{1}{2}$ then the space $\Sigma_s(\rr d)$ is reflexive. 
\end{prop}

\begin{proof}
By \cite[Exercise~V.52]{Reed1} the Fr\'echet space $\Sigma_s(\rr d)$ carries the Mackey topology. 
By \cite[Exercise~V.56 (a) and Lemma on p.~166]{Reed1} it remains to prove the following statement: 
Every weakly closed and weakly bounded subset $B \subseteq \Sigma_s(\rr d)$ is weakly compact. 

By \cite[Theorem~V.23]{Reed1} $B \subseteq \Sigma_s(\rr d)$ is bounded in the Fr\'echet space topology of $\Sigma_s(\rr d)$. 
The Fr\'echet space topology on $\Sigma_s(\rr d)$ is stronger than the weak topology. 
This fact implies that $B \subseteq \Sigma_s(\rr d)$ is closed in its Fr\'echet space topology, 
and if $B$ is shown to be compact in the Fr\'echet space topology then it is also weakly compact. 

Thus it remains to show that $B \subseteq \Sigma_s(\rr d)$ is compact in its Fr\'echet space topology. 
Since the Fr\'echet space topology of $\Sigma_s(\rr d)$ is metric  
we may prove compactness of $B$ by showing that any sequence $( f_n )_{n \geqs 1} \subseteq B$ has a convergent subsequence. 
The space $\Sigma_s(\rr d)$ is complete and $B$ is closed so it is suffices to show the existence of a Cauchy subsequence of $(f_n)_{n \geqs 1} \subseteq \Sigma_s(\rr d)$.
We accomplish this by constructing a subsequence which is Cauchy in the seminorm \eqref{eq:seminormSigmas} for the space ${\mathcal S_{s,h}}$ for each $0 < h \leqs 1$. 

We have since $B \subseteq \Sigma_s(\rr d)$ is bounded
\begin{equation}\label{eq:fnbound}
| x^\alpha D^\beta f_{n}(x)| \leqs C_h h^{|\alpha+\beta|} (\alpha! \beta!)^{s} , \quad x \in \rr d, \quad \alpha, \beta \in \nn d, \quad n \geqs 1, \quad h > 0, 
\end{equation}
for some constants $C_h > 0$. 

Let $0 < h \leqs 1$ and let $\ep > 0$. 
The bound \eqref{eq:fnbound} gives
\begin{align*}
| x^\alpha D^\beta f_n (x) |
& = |x|^{-2} \left| \sum_{j=1}^d x_j^2 x^\alpha D^\beta f_n(x) \right| \\
& \leqs |x|^{-2} C_{\frac{h}{2}} \sum_{j=1}^d \left( \frac{h}{2} \right)^{|\alpha+\beta| + 2} (\alpha! (\alpha_j+1) (\alpha_j+2) \beta! )^s \\
& \leqs |x|^{-2} C_{\frac{h}{2}} \, \left( \frac{h}{2} \right)^{|\alpha+\beta|} (\alpha! \beta!)^s d (|\alpha|+2)^{2s} \\
& \leqs |x|^{-2} C \, C_{\frac{h}{2}} \, h^{|\alpha+\beta|} (\alpha! \beta!)^s, \quad x \in \rr d \setminus 0, \quad \alpha, \beta \in \nn d, \quad n \geqs 1,
\end{align*}
for some $C > 0$. 
This gives 
\begin{equation}\label{eq:boundlargex}
\sup_{\alpha, \beta \in \nn d, \ |x| \geqs L} 
\frac{|x^\alpha D^\beta ( f_n(x) - f_m(x))| }{h^{|\alpha+\beta|} (\alpha! \beta!)^{s}} < \ep, \quad n,m \geqs 1, 
\end{equation}
provided $L > 0$ is sufficiently large.

Next we consider the sequences $(x^\alpha D^\beta f_n(x))_{n \geqs 1}$ for $|\alpha + \beta| > N$ where $N \in \no$ is to be chosen. 
Again \eqref{eq:fnbound} yields
\begin{align*}
| x^\alpha D^\beta f_{n}(x)| 
& \leqs C_{\frac{h}{2}}  \left( \frac{h}{2} \right)^{|\alpha+\beta|}  (\alpha! \beta!)^s \\
& \leqs C_{\frac{h}{2}} 2^{-N} h^{|\alpha+\beta|} (\alpha! \beta!)^{s}, \quad x \in \rr d, \quad |\alpha + \beta| > N, \quad n \geqs 1, 
\end{align*}
which proves the estimate
\begin{equation}\label{eq:boundsmallindices}
\sup_{|\alpha+\beta| > N, \  x \in \rr d} 
\frac{|x^\alpha D^\beta ( f_n(x) - f_m(x))| }{h^{|\alpha+\beta|} (\alpha! \beta!)^{s}} < \ep, \quad n,m \geqs 1, 
\end{equation}
provided $N \in \no$ is sufficiently large.

Finally we study the sequences of functions $( x^\alpha D^\beta f_n(x) )_{n \geqs 1}$ restricted to the compact ball ${\overline B}_L = \{ x \in \rr d: \ |x| \leqs L \}$, where $\alpha,\beta \in \nn d$ satisfy $|\alpha + \beta| \leqs N$.
If $1\leqs j \leqs d$ we obtain from \eqref{eq:fnbound} if $\alpha_j = 0$
\begin{align*}
| D_j  (x^\alpha D^\beta f_n )(x) |
= | x^\alpha D^{\beta+e_j} f_n (x) |   
& \leqs C_h h^{|\alpha+\beta|+1} (\alpha! \beta!)^{s} (|\beta|+1)^s \\
& \leqs C_h (\alpha! \beta!)^{s} (|\beta|+1)^s , \quad x \in \rr d, 
\end{align*}
and if $\alpha_j > 0$
\begin{align*}
| D_j  (x^\alpha D^\beta f_n )(x) |
& = | i^{-1} \alpha_j x^{\alpha-e_j} D^\beta f_n (x)  + x^\alpha D^{\beta+e_j} f_n (x)| \\ 
& \leqs C_h (\alpha! \beta!)^{s} (|\alpha| h ^{|\alpha+\beta| -1} + h ^{|\alpha+\beta|+1} (|\beta|+1)^s) \\
& \leqs C_h (\alpha! \beta!)^{s} (|\alpha|  + (|\beta|+1)^s ) , \quad x \in \rr d. 
\end{align*}

The gradient is thus uniformly bounded with respect to $x \in \rr d$: 
\begin{equation*}
\sup_{x \in \rr d} |\nabla (x^\alpha D^\beta f_n ) (x)| \leqs C_{h,\alpha,\beta} < \infty.
\end{equation*}
The mean value theorem gives
\begin{equation*}
| (x^\alpha D^\beta f_n ) ( x) - (x^\alpha D^\beta f_n ) (y)|
\leqs C_{h,\alpha,\beta} | x - y |, \quad x,y \in \rr d, 
\end{equation*}
which shows that $\{ x^\alpha D^\beta f_n, \ n \geqs 1  \}$ is an equicontinuous set of functions on $\rr d$ for all $\alpha,\beta \in \nn d$, 
particularly if $|\alpha + \beta| \leqs N$. 

Combining with the bound \eqref{eq:fnbound} which is uniform with respect to $x \in \rr d$ and $n \geqs 1$ we find that the assumptions for the Arzel{\` a}--Ascoli theorem \cite[Theorem~11.28]{Rudin1} are satisfied for 
$\{ x^\alpha D^\beta f_n, \ n \geqs 1 \}$, for each $\alpha,\beta \in \nn d$. 

Thus we start by extracting a subsequence of $(f_n)_{n \geqs 1}$ that converges uniformly on ${\overline B}_L$. 
We apply $x^\alpha D^\beta$ to the subsequence and extract a new subsequence that converges uniformly on ${\overline B}_L$, 
consecutively, first for all $\alpha,\beta \in \nn d$
such that $|\alpha+\beta|=1$, and after that for all multi-indices $\alpha,\beta \in \nn d$ of increasing orders $|\alpha+\beta| = 2, \dots, N$. 
After a finite number of such subsequence extractions we obtain a subsequence $(f_{n_k})_{k \geqs 1}$ such that 
\begin{equation}\label{eq:boundcompactballsmallindices}
\sup_{|\alpha+\beta| \leqs N, \ |x| \leqs L} 
\frac{|x^\alpha D^\beta ( f_{n_k}(x) - f_{n_m}(x))| }{h^{|\alpha+\beta|} (\alpha! \beta!)^{s}} < \ep, \quad k,m \geqs K, 
\end{equation}
provided $K \in \no$ is sufficiently large. 
When we combine \eqref{eq:boundlargex}, \eqref{eq:boundsmallindices} and \eqref{eq:boundcompactballsmallindices}
it follows that $(f_{n_k})_{k \geqs 1}$ is a Cauchy sequence in ${\mathcal S_{s,h}}$.
\end{proof}

\section*{Acknowledgment}
Work partially supported by the MIUR project ``Dipartimenti di Eccellenza 2018-2022'' (CUP E11G18000350001).

%%%%%%%%%%%%%%%%%%%%%%%%%%%%%

\end{document}